\documentclass[11pt]{article}
\usepackage{latexsym,amsmath,amsthm}
\usepackage{amssymb}
\usepackage{graphicx}
\usepackage{appendix}
\usepackage{tikz}
\usepackage[colorlinks=true, linkcolor=blue]{hyperref}
\everymath{\displaystyle}

\newtheorem{thm}{Theorem}[section]

\newtheorem{cor}[thm]{Corollary}
\newtheorem{lemma}[thm]{Lemma}
\newtheorem{conj}[thm]{Conjecture}

\theoremstyle{definition}

\theoremstyle{definition}

\theoremstyle{definition}
\newtheorem{defn}[thm]{Definition}

\theoremstyle{definition}
\newtheorem{ex}[thm]{Example}

\theoremstyle{definition}

\theoremstyle{definition}

\theoremstyle{remark}

\theoremstyle{remark}
\newtheorem{remark}[thm]{Remark}

\usepackage[margin=1.3in]{geometry}

\usepackage{url,xcolor}
\usepackage{microtype}

\newcommand{\NN}{\mathbb{N}}
\newcommand{\Sym}{\mathbf{Sym}}

\begin{document}

\title{Reciprocals of thinned exponential series} 

\author{John Engbers\thanks{Department of Mathematical and Statistical Sciences, Marquette University, Milwaukee, WI; john.engbers@marquette.edu. Research supported by the Simons Foundation Gift number 524418.} \and David Galvin\thanks{Department of Mathematics,
University of Notre Dame, Notre Dame IN; dgalvin1@nd.edu. Research supported in part by the Simons Foundation Gift number 854277.} \and Cliff Smyth\thanks{Department of Mathematics and Statistics, University of North Carolina at Greensboro, Greensboro NC; cdsmyth@uncg.edu. Research supported by the Simons Foundation Gifts 360468 and 965562.}}

\maketitle

\begin{abstract}

The reciprocal of $e^{-x}$ has a power series about $0$ in which all coefficients are non-negative. Gessel \cite{Gessel} considered truncates of the power series of $e^{-x}$, i.e. polynomials of the form $\sum_{n=0}^r (-1)^n\frac{x^n}{n!}$, and established combinatorially that the reciprocal of the truncate has a power series with all coefficients non-negative precisely when $r$ is odd.
    
Here we extend Gessel's observations to arbitrary ``thinned exponential series''.
To be precise, let $A \subseteq \{1,3,5,\ldots\}$ and $B \subseteq \{2,4,6,\ldots\}$, and consider the series  
\[
1-\sum_{a \in A} \frac{x^a}{a!} + \sum_{b \in B} \frac{x^b}{b!}.
\]
We consider conditions on $A$ and $B$ that ensure that the reciprocal series has all coefficients non-negative.  We give combinatorial proofs for a large set of conditions, including whenever $1 \in A$ and the endpoints of the maximal consecutive intervals in $A \cup B$ are odd integers. 
 
In particular, the coefficients in the reciprocal series can be interpreted as ordered set partitions of $[n]$ with block size restrictions, or in terms of permutations with restricted lengths of maximally increasing runs, suitably weighted.

\end{abstract}

\section{Introduction and Statement of Results}

The power series $\sum_{n=0}^\infty (-1)^n \frac{x^n}{n!}$ 
has a reciprocal with only non-negative coefficients; specifically, for real $x$, we have 
$$
\left(\sum_{n=0}^\infty (-1)^n \frac{x^n}{n!}\right)^{-1} = \left(e^{-x}\right)^{-1} = e^x = \sum_{n=0}^\infty \frac{x^n}{n!}.
$$
Gessel \cite{Gessel} observed that whether this non-negativity phenomenon extends to the reciprocals of partial sums of $\sum_{n=0}^\infty (-1)^n \frac{x^n}{n!}$ seems to depend on the parity of the point of truncation --- in particular, computation suggests that for all $m \in {\mathbb N}$,  $\left(\sum_{n=0}^{2m} (-1)^n \frac{x^n}{n!}\right)^{-1}$ has some negative coefficients, while $\left(\sum_{n=0}^{2m-1} (-1)^n \frac{x^n}{n!}\right)^{-1}$ has only non-negative coefficients. Gessel gives a lovely combinatorial proof of this latter fact, demonstrating that if the sequence $(c_n)_{n \geq 0}$ is defined via \[
\left(\sum_{n=0}^{2m-1} (-1)^n \frac{x^n}{n!}\right)^{-1} = \sum_{n \geq 0} c_n \frac{x^n}{n!},
\]
then $c_n$ counts the number of permutations of $[n]:=\{1, \ldots, n\}$ in which
every increasing run has length congruent to $0$ or $1$ modulo $2m$. This is a special case ($b=1, r=2$) of the following more general result from \cite{Gessel}:
\begin{thm} (\cite[Proposition 2.4]{Gessel}) \label{thm-Gessel}
Let $m$, $b$ and $r \geq 2$ be positive integers. The coefficient of $x^n/n!$ in
$$
\left(\sum_{k=0}^{m-1} \left(\frac{x^{krb}}{(krb)!}-\frac{x^{(kr+1)b}}{((kr+1)b)!}\right)\right)^{-1}
$$
is the number of permutations of $[n]$ in which every maximal increasing run has length congruent to one of $0, b, \ldots, (r-1)b$ modulo $mrb$.
\end{thm}

In particular, the coefficients in the reciprocal series are all non-negative.

Inspired by this result, we take a combinatorial approach to extend the scope of Gessel's observations. 

\medskip

To state our results precisely, let $\mathcal{O}=\{1,3,5,\ldots\}$ denote the odd natural numbers and $\mathcal{E} = \{2,4,\ldots\}$ denote the positive even natural numbers, and let $A \subseteq \mathcal{O}$ and $B \subseteq \mathcal{E}$. Consider the power series
\[
F_{A,B}(x) = 1- \sum_{a \in A} \frac{x^a}{a!} + \sum_{b \in B}  \frac{x^b}{b!},
\]
and let 
\[
G_{A,B}(x) = (F_{A,B}(x))^{-1} = \sum_{n=0}^{\infty} c_n \frac{x^n}{n!}
\]
be the reciprocal of $F_{A,B}(x)$.  A natural question to ask is: under what conditions does the power series representation of $G$ about $0$ have all coefficients $c_n$ non-negative?  And when $G$ does have all coefficients $c_n$ non-negative, is there a natural set of objects that $c_n$ counts? For results on these questions and their combinatorial interpretations in the context of {\em compositional} inverses, see e.g. \cite{EGS}.

If we are content to interpret $c_n$ as the number of elements in a signed set (a set in which each element has weight either $1$ or $-1$), or, equivalently, as the difference between the sizes of two sets, then the matter is straightforward, as we show next.  Write ${\mathcal P}^{\rm pos}_{n,A,B}$ for the set of ordered set partitions $P=P_1/\cdots/P_\ell$ of $[n]$, in which all block sizes are in $A \cup B$, and which have an even number of blocks with sizes from $B$.  Write ${\mathcal P}^{\rm neg}_{n,A,B}$ for those with an odd number of blocks with sizes from $B$. 
Note that 
\[
G_{A,B}(x) = \frac{1}{F_{A,B}(x)} = \sum_{k=0}^{\infty} \left( \sum_{a \in A} \frac{x^{a}}{a!} - \sum_{b \in B} \frac{x^{b}}{b!} \right)^k.
\]
The coefficient of $x^n/n!$ in the final expression above is a sum over compositions of $n$ with all blocks coming from $A\cup B$; the sign of the summand corresponding to a composition is positive if the composition has an even number of blocks from $B$, and negative otherwise, and the magnitude is the multinomial coefficient $\binom{n}{a_1, a_2, \ldots, b_1, b_2, \ldots}$, where the $a_i$'s and $b_i$'s are the blocks of the composition. It immediately follows that we have that the coefficient of $x^n/n!$ in $(F_{A,B}(x))^{-1}$ is
\[
|\mathcal{P}^{\rm pos}_{n,A,B}| - |\mathcal{P}^{\rm neg}_{n,A,B}|.
\]

The drawback to this interpretation is that it gives no immediate information about the sign of the coefficient, nor any obvious combinatorial explanation of why the sign might be always non-negative for certain choices of $A$ and $B$. One strategy for showing that $G_{A,B}(x)$ has all coefficients non-negative is to exhibit an injection $i$ from ${\mathcal P}^{\rm neg}_{n,A,B}$ into ${\mathcal P}^{\rm pos}_{n,A,B}$ or, equivalently, an involution $\tau$ of the restricted partitions in ${\mathcal P}^{\rm neg}_{n,A,B} \cup {\mathcal P}^{\rm pos}_{n,A,B}$ that moves all elements in ${\mathcal P}^{\rm neg}_{n,A,B}$ to elements in ${\mathcal P}^{\rm pos}_{n,A,B}$ and fixes the elements in ${\mathcal P}^{\rm pos}_{n,A,B} \setminus \tau({\mathcal P}^{\rm neg}_{n,A,B})$. 

\medskip

As discussed earlier, Gessel \cite{Gessel} observes that in particular when $A \cup B=\{1, \ldots 2m-1\}$ then $G_{A,B}(x)$ has all coefficients non-negative. We obtain a broad generalization of this (Theorem \ref{thm-odd-ended} below). We will need the following definitions.

\begin{defn}
For $A \subseteq  \mathcal{O}$ and $B \subseteq \mathcal{E}$, we say that $A \cup B$ is {\em odd-ended} if, when it is written as the union of maximal intervals of consecutive integers, the endpoints of each interval are odd numbers.  A maximal element in a maximal interval is called a {\em top element} and a minimal element in a maximal interval is called a {\em bottom element}.
\end{defn}
 
\begin{ex}
The set $\{1, \ldots, 2m-1\}$ is odd-ended, and only contains one maximal interval, with top element $2m-1$ and bottom element $1$.

The set $\{1,5,9,10,11,12,13,19,20,21,22,\ldots\}=\{1\}\cup\{5\}\cup[9,13]\cup[19,\infty)$ is odd-ended and has four maximal intervals.  The top elements are $1$, $5$, and $13$, while the bottom elements are $1$, $5$, $9$, and $19$.
\end{ex}

Our first main result is as follows.
\begin{thm} \label{thm-odd-ended}
Suppose that $A \subseteq \mathcal{O}$ and $B \subseteq \mathcal{E}$ with $A \cup B$ odd-ended and $1 \in A$.  

Define $(c_n)_{n \geq 0}$ via 
\[
\left(1-\sum_{a \in A} \frac{x^a}{a!} + \sum_{b \in B} \frac{x^b}{b!} \right)^{-1} = \sum_{n \geq 0} c_n \frac{x^n}{n!}.
\]
Then the values $c_n$ are all non-negative integers, and moreover have a natural combinatorial interpretation as the number of ordered set partitions of $[n]$ whose blocks are either size $1$ or size that is an endpoint of the maximal intervals, and these block sizes appear in particular allowable orders as described in Corollary \ref{cor-partitioninterp}. 
\end{thm}

In Corollary \ref{cor-perminterp}, we illustrate that these coefficients count the permutations with particular increasing run lengths, suitably weighted, which extends the results of Gessel \cite{Gessel}.   The coefficients $c_n$ for $A \cup B = [3]$ are given as sequence A317111 in OEIS.  Those for $A \cup B = [5]$, $[7]$, and $[9]$ are given as sequences A322262, A322282, and A322283 in OEIS, respectively \cite{OEIS}.  Those for $A \cup B = \{1,3\}$ or $A \cup B = \{1,2,3,5\}$ are simple examples where $A \cup B$ is odd-ended whose corresponding coefficient sequences do not appear in OEIS.

\medskip

We also incorporate the integer parameters $b \geq  1$ and $r \geq 2$ into our results, which we refer to as the {\em $b$-stretched} and {\em $r$-stretched} versions of our main result.  The idea is to stretch the set $\{1,2,3,4,\ldots\}$ to $\{1,r,r+1,2r,\ldots\}$, and then multiply each entry by $b$ to obtain $\{b,br,b(r+1),2br,\ldots\}$.  We then suppose that $A \subseteq {\mathcal O}^* := \{b,b(r+1),\ldots \}$ (the stretched odds) and $B \subseteq {\mathcal E}^* := \{br,2br,\ldots\}$ (the stretched evens).  We will again say that $A \cup B$ is {\em odd-ended} if the endpoints of the (stretched) maximal intervals again are ``odd''; 
in particular $A \cup B$ is odd-ended if $kr \in B$ implies $kr+1 \in A$ and $(k-1)r+1 \in A$.  We make similar modifications to definitions throughout in the stretched context.

\begin{thm}\label{thm-odd-endedstretched}
Suppose that $A \subseteq {\mathcal O}^*$ and $B \subseteq {\mathcal E}^*$ with $b \in A$ and $A \cup B$ odd-ended. Define $(c_n)_{n \geq 0}$ via 
\[
\left(1-\sum_{a \in A} \frac{x^a}{a!} + \sum_{b \in B} \frac{x^b}{b!} \right)^{-1} = \sum_{n \geq 0} c_n \frac{x^n}{n!}.
\]
Then the values $c_n$ are all non-negative integers, and moreover have a natural combinatorial interpretation as the number of ordered set partitions of $[n]$ whose blocks are either size $1$ or size that is an endpoint of the stretched maximal intervals, and these block sizes appear in particular allowable orders as stated in Corollary \ref{cor-partitioninterpstretched}.
\end{thm}
In Corollary \ref{cor-perminterpstretched}, we illustrate that these coefficients count the permutations with particular increasing run lengths, suitably weighted.

We give an alternate proof Theorem \ref{thm-odd-ended} and Theorem \ref{thm-odd-endedstretched} in Section \ref{sec-run-theorem}. That proof apeals to the ``Run Theorem'' in non-commutative symmetric function theory (Theorem \ref{thm-run-theorem}  \cite{Gesselthesis}) as well as some algebraic manipulations of generating functions.  In contrast, the proofs of Theorems \ref{thm-odd-ended} and \ref{thm-odd-endedstretched} given in Sections \ref{sec-pfofodd-ended} and \ref{sec-stretched} are fully combinatorial. Those proofs also involve sign-reversing involutions, a tactic used in \cite{EGS} to give non-negativity results and combinatorial formulas for the coefficients of the \emph{compositional inverses} of a large class of power series. This technique may be useful in proving non-negativity of coefficients in other situations as well.

Note that the case $b=1$ and $r=2$ in Theorem \ref{thm-odd-endedstretched} is simply Theorem \ref{thm-odd-ended}. The specifics of our combinatorial interpretations, which are direct extensions of the $r=2$ version, are stated in Corollaries \ref{cor-partitioninterpstretched} and \ref{cor-perminterpstretched}. Extending the results of Theorem \ref{thm-odd-ended} to $b>1$ is straight-forward, while extending to $r>2$ requires a bit more care. 

When we consider the $r$-stretched version from Theorem \ref{thm-odd-endedstretched} with $b=1$, in the special case where $A \cup B = \{1,r,r+1,2r,\ldots\}$, we can interpret 
\[
\frac{1}{1-x+\frac{x^{r}}{r!}-\frac{x^{r+1}}{(r+1)!}+\frac{x^{2r}}{(2r)!}-\frac{x^{2r+1}}{(2r+1)!}+\cdots}
\]
as a series whose coefficients count the number of permutations on $[n]$ so that the increasing runs have length at most $r-1$, which recovers a result of David and Barton (\cite[pp. 156-157]{DavidBarton}, see also for example \cite{GesselZhuang}). See Section \ref{subsec-GesselStretched}.

\medskip

Gessel considers the following more general situation (see \cite[Theorem 3.3]{Gessel}). 
Suppose $a(x) = a_0 + a_1x + \cdots + a_dx^d$ with $a_0=1$, and further suppose that $1/a(x)$ is an (ordinary) power series with all coefficients $0$ or $1$.
Then the coefficients of the reciprocal of the {\em exponential} series $1/(a_0 + a_1x/1! + \cdots + a_dx^d/d!)$ are non-negative, and the coefficients have a combinatorial interpretation in terms of lengths of permutation runs. Further, there is a necessary and sufficient condition for such $a(x)$ to meet the $0$ and $1$ reciprocal coefficients property \cite[Theorem 3.2]{Gessel}. Our Theorem \ref{thm-odd-ended} gives an interpretation when $A \cup B = \{1,5,6,7\}$, that is, for the coefficients for the expression $1/(1-x-x^5/5!+x^6/6!-x^7/7!)$. This is not covered by Gessel's theorem, as $$\frac{1}{1-x-x^5+x^6-x^7} = 1+x+x^2+x^3+x^4+2x^5+\cdots$$ and so not every coefficient is a $0$ or $1$. Of course, as we observe in Section \ref{sec-run-theorem}, cases like this can also be recovered from the Run Theorem.

\medskip

Returning to the non-stretched setting, an obvious question is to determine necessary and sufficient conditions (on $A, B$) for $G_{A,B}(x)$ to be non-negative.  We show a necessary condition, under the requirement that $1 \in A$, in the following. Recall that when $A \cup B = \{1,2,3,4,\ldots\}$, equivalently when  $\{1,2,3,\ldots\} \setminus (A \cup B) = \varnothing$, the reciprocal has all coefficients non-negative.

\begin{thm}\label{thm-necessary}
Suppose that $A \subseteq \mathcal{O}$ and $B \subseteq \mathcal{E}$ with $1 \in A$. 
Define $(c_n)_{n \geq 0}$ via 
\[
\left(1-\sum_{a \in A} \frac{x^a}{a!} + \sum_{b \in B} \frac{x^b}{b!} \right)^{-1} = \sum_{n \geq 0} c_n \frac{x^n}{n!}.
\]
If the smallest positive number in $\{1,2,3,\ldots\} \setminus (A \cup B)$ is $2m+1$, then the coefficient $c_{2m+2}$ is negative.
\end{thm}
In other words, if the initial interval in $A \cup B$ ends in an even number, then the reciprocal has a negative coefficient. This provides a combinatorial proof of Gessel's initial observation in \cite{Gessel} when $A \cup B = [2m]$.  

Computational evidence suggests that (still in the special case $1 \in A$) if just the initial interval ends in an odd number, then the reciprocal has non-negative coefficients.  

\begin{conj}\label{conj-sufficient}
Suppose that $A \subseteq \mathcal{O}$ and $B \subseteq \mathcal{E}$ with $1 \in A$.  

Define $(c_n)_{n \geq 0}$ via 
\[
\left(1-\sum_{a \in A} \frac{x^a}{a!} + \sum_{b \in B} \frac{x^b}{b!} \right)^{-1} = \sum_{n \geq 0} c_n \frac{x^n}{n!}.
\]
If the smallest positive number in $\{1,2,3,\ldots\} \setminus (A \cup B)$ is $2m$, then all coefficients $c_n$ are non-negative.
\end{conj}

We are able to reduce Conjecture \ref{conj-sufficient} to the particular case where $A \cup B = [2m-1] \cup \{2m+2,2m+4,\ldots\}$ for some positive integer $m$; see Corollary \ref{cor-sufficientreduce}.  We then use analytic arguments to prove Conjecture \ref{conj-sufficient} in the cases where $m=1$ and $m=2$, see Theorems \ref{thm-analytic} and \ref{thm-analytic2}.

These analytic techniques do not seem to support extending this argument much further, as we discuss in Section \ref{sec-necsuff} after the proofs. Further evidence in favor of Conjecture \ref{conj-sufficient} is provided by computation that shows that if $m \leq 250$ then $c_n \geq 0$ for $n \leq 2500$. 

Note that while the Run Theorem, Theorem \ref{thm-run-theorem}, can be used to prove Theorems \ref{thm-odd-ended} and \ref{thm-odd-endedstretched}, it cannot be used to resolve this conjecture.  See Section \ref{sec-run-theorem} for the details.

Our combinatorial proofs all rely on an involution on the ordered set partitions of $[n]$ that {\em fixes the underlying permutation} and associates a non-negative coefficient to each fixed permutation; see Section \ref{sec-proofideas} for an overview of this involution. We can in fact show that our method of proof has particular limitations that hinder its ability to prove the sufficient condition found in Conjecture \ref{conj-sufficient}.  See Section \ref{sec-necsuff} for the details.

\medskip

The rest of the paper is laid out as follows. In Section \ref{sec-proofideas}, we give a high-level overview of the algorithm which is used throughout the paper, and illustrate the algorithm on several explicit partitions. Then, in Section \ref{sec-pfofodd-ended}, we give the proof of Theorem \ref{thm-odd-ended} as well as the combinatorial interpretations of the coefficients $c_n$ of the reciprocal.  Section \ref{sec-stretched} extends the proof to the stretched setting, giving both the proof of Theorem \ref{thm-odd-endedstretched} and the corresponding combinatorial interpretations.  We end in Section \ref{sec-necsuff} with the proof of Theorem \ref{thm-necessary} along with the reasons why our algorithm is unable to prove Conjecture  \ref{conj-sufficient} in full. We then use an analytic argument to reduce Conjecture \ref{conj-sufficient} to a particular $A\cup B$ for each positive integer $m$, and use that reduction to analytically prove the conjecture for $m=1$ and $m=2$; we also give evidence that the natural adjustments to our analytic argument are unable to fully resolve Conjecture \ref{conj-sufficient} for all positive integers $m$.  We conclude with Section \ref{sec-run-theorem} in which we show how the Run Theorem in non-commutative symmetric function theory along with some algebraic manipulations can be used to provide alternate proofs to Theorems \ref{thm-odd-ended} and \ref{thm-odd-endedstretched}.

\section{Proof Ideas}\label{sec-proofideas}

In this section, we briefly sketch an algorithmically defined involution $i$, that is involved in proving Theorems \ref{thm-odd-ended} and \ref{thm-odd-endedstretched}.  We then illustrate how it works on some inputs.

Given a partition $P$ of $[n]$ whose elements in a block 
are listed in increasing order, the algorithm attempts to perform one of two operations:
\begin{itemize}
    \item if the block furthest to the right has a single element, it attempts to merge it with the one to its left (subject to some conditions), or 
    \item if the block furthest to the right has more than one element, it splits off its furthest right element into a new singleton block.  
\end{itemize}
If it cannot do either of these (e.g. because they will create blocks of forbidden sizes), it will move to the next block of $P$ to the left; occasionally it will need to also skip this next block too. It takes some care to ensure the involution is well-defined and prove that it has the properties we need.  We will do this in Section 3, but first we illustrate the involution operating on some particular inputs.

\begin{ex}
Suppose $A=\{1,3\}$, $B=\{2\}$, and $n=9$.  This means we have ordered partitions of the set $[9]$ into blocks, with the elements inside of a block written in increasing order.  Furthermore, the blocks can only contain $1$, $2$, or $3$ elements.

\begin{itemize}
    \item The partition $P=234/6/1/57/89$ of $[9]$ will map to $234/6/1/57/8/9$, and this latter partition will map back to the former partition $234/6/1/57/89$.  
    
    For future reference, we also note here that the underlying permutation for $P$ is $234615789$, and this is the same as the underlying partition for the image of $P$. This property will always hold for our algorithm.

    \item The partition $P=234/1/56/78/9$ will map to $234/1/56/789$, and this latter partition will map back to the former partition.

    \item The partition $P=234/1/56/79/8$ will skip the first singleton (as the entries of merging blocks $79$ and $8$ are not in order), and then map to $234/1/56/7/9/8$. This latter partition will again skip the first singleton (as the entries of merging blocks $9$ and $8$ are out of order) and then map back to the former partition. 

    \item The partition $P=134/28/567/9$ will skip the first singleton (since blocks of size $4$ are not allowed).  It will then also \emph{skip the block $567$} and move to considering block $28$.  
    
    The reason for this somewhat surprising second skip, briefly, is that if not, the partition $P$ will instead map to $134/28/56/7/9$, and this latter partition will \emph{not} map back to the former under this procedure (it will map to $134/28/56/79$ to be consistent with previously defined rules, and so our algorithm will not be an involution). Therefore, to make this algorithm an involution, we also skip the block $567$.
\end{itemize}
\end{ex}

We also note that being odd-ended means that any even-sized block will be able to either add or subtract a singleton and still have a block size in $A \cup B$, thus mapping elements of ${\mathcal P}^{\rm neg}_{n,A,B}$ to elements of ${\mathcal P}^{\rm pos}_{n,A,B}$. 

Finally, we briefly indicate how the procedure will change when considering the values $b >1$ and $r>2$. When we have $b>1$, we will view each individual number in $[n]$ as a group of size $b$, and when we have a value $r>2$ we will translate from the set $\{1,2,3,4,5,\ldots\}$ to the set $\{1,r,r+1,2r,2r+1,\ldots\}$.

\section{Proof of Theorem \ref{thm-odd-ended} and Combinatorial Interpretations}\label{sec-pfofodd-ended}

In this section, we prove Theorem \ref{thm-odd-ended} and provide the combinatorial interpretations of the coefficients of the reciprocal series.  We begin with the proof.

\subsection{Proof of Theorem \ref{thm-odd-ended}}
Suppose that $A \subseteq {\mathcal O}$ and $B \subseteq {\mathcal E}$ satisfy that $1 \in A$ and $A \cup B$ is odd-ended. 

For each non-negative integer $n$ we construct an injection from ${\mathcal P}^{\rm neg}_{n,A,B}$ into ${\mathcal P}^{\rm pos}_{n,A,B}$ 
by constructing inductively for each $n \geq 0$ a sign-reversing involution $i_n$ on ${\mathcal P}^{\rm pos}_{n,A,B} \cup {\mathcal P}^{\rm neg}_{n,A,B}$, all of whose fixed points lie in ${\mathcal P}^{\rm pos}_{n,A,B}$. By sign-reversing, we mean that if $i(P) \neq P$ then $i(P) \in  {\mathcal P}^{\rm pos}_{n,A,B}$ if $P \in {\mathcal P}^{\rm neg}_{n,A,B}$, and $i(P) \in  {\mathcal P}^{\rm neg}_{n,A,B}$ if $P \in {\mathcal P}^{\rm pos}_{n,A,B}$. 
Furthermore, we will show that the underlying permutation on $i_n(P)$ is the same as the underlying permutation on $P$ (i.e. the relative order of the elements of $[n]$ in the ordered partition is not changed by $i_n$). The inductive procedure can equivalently be thought of as an iterative algorithm on a given ordered partition. 
\medskip

\noindent {\bf Constructing $i_0$}: When $n=0$ there is a unique ordered partition, the empty partition, which is in ${\mathcal P}^{\rm pos}_{0, A,B}$ as it has zero blocks with sizes from $B$. The involution is the identity map, and the relative order condition is trivially satisfied.

\medskip

\noindent {\bf Constructing $i_1$}: When $n=1$ there is again a unique ordered partition, with one singleton block $\{1\}$, which is in ${\mathcal P}^{\rm pos}_{1,A,B}$, as it has zero blocks with sizes from $B$. The involution is again the identity map, and also the relative order condition is trivially satisfied. 

\medskip

\noindent {\bf Constructing $i_n$ for $n \geq 2$}: We take as our inductive hypothesis that for $m < n$, a sign-reversing involution $i_m$  on ${\mathcal P}^{\rm pos}_{m, A,B} \cup {\mathcal P}^{\rm neg}_{m, A,B}$, all of whose fixed points lie in ${\mathcal P}^{\rm pos}_{m, A,B}$, has been constructed with the property that for every $P \in {\mathcal P}^{\rm pos}_{m, A,B} \cup {\mathcal P}^{\rm neg}_{m,A,B}$, the permutation underlying $i_m(P)$ is the same as that underlying $P$. 

We first construct an explicit involution $i_n$  on ${\mathcal P}^{\rm pos}_{n, A,B} \cup {\mathcal P}^{\rm neg}_{n, A,B}$, that also satisfies the relative order condition.

Suppose that $P:=P_1/P_2/\cdots/P_\ell$ is an ordered partition of $[n]$ with $|P_i| \in A \cup B$ for $i=1,\ldots,\ell$, where the elements inside each $P_i$ are listed in increasing order.

\bigskip

\noindent \hrulefill

\noindent \textbf{ALGORITHM: Case $r=2$}

\begin{description}
\item[(1)] Suppose $|P_\ell|=1$.  Say block $P_\ell$ consists of the single number $k$: 
\begin{description}
\item[Case A] If $k$ is larger than the largest entry (the right-most entry) of block $P_{\ell-1}$, and if also $|P_{\ell-1}|+1 \in A\cup  B$, then do a {\em merge move}:  combine blocks $P_{\ell-1}$ and $P_\ell$, so $i_n(P) = P_1/P_2/\cdots/P_{\ell-1}P_\ell$.

\item[Case B] If $k$ is smaller than the largest entry in $P_{\ell-1}$, then do a {\em block skip}: fix $P_\ell$ and next consider $P_{\ell-1}$, so $i_n(P)=i_{n-1}(P_1/P_2/\cdots/P_{\ell-1})/P_\ell$.\footnote{Since $[n]\setminus\{k\}$ is not (necessarily) equal to $[n-1]$, what we mean here is: Let $P'=P_1/P_2/\cdots/P_{\ell-1}$ (an ordered partition of $[n]\setminus\{k\}$). Let $\varphi$ be the unique order-preserving map from $[n]\setminus\{k\}$ to $[n-1]$. This naturally induces a map $\varphi'$ from ordered partitions of $[n]\setminus\{k\}$ to ordered partitions of $[n-1]$. Then set  $i_n(P)=(\varphi')^{-1}(i_{n-1}(\varphi'(P')))/P_\ell$.}  
\item[Case C] If $k$ is larger than the largest entry in $P_{\ell-1}$, and if also $|P_{\ell-1}|+1 \notin A \cup B$, then do a {\em freeze skip}: fix {\em both} $P_\ell$ and $P_{\ell-1}$, and next consider $P_{\ell-2}$, so $i_n(P)=i_{n-1-|P_{\ell-1}|}(P_1/P_2/\cdots/P_{\ell-2})/P_{\ell-1}/P_\ell$.\footnote{To make this precise requires similar machinations as those required for Case B.}
\end{description}
\item[(2)] Suppose $|P_\ell|>1$, with $k$ the largest element of $P_\ell$:
\begin{description}
\item[Case D] If $|P_\ell|-1 \in A \cup B$, then make a {\em split move}: replace the block $P_\ell$ with two blocks, $P_\ell-k$ and $k$, in that order, so $i_n(P)=P_1/P_2/\cdots/P_{\ell-1}/P_\ell-k/k$. 
\item[Case E] If $|P_{\ell}|-1 \notin A\cup B$, then do a block skip, so $i_n(P) = i_{n-|P_\ell|}(P_1/P_2/\cdots/P_{\ell-1})/P_\ell$\footnote{To make this precise requires similar machinations as those required for Cases B and C.}. 
\end{description}

\end{description}  
\noindent \hrulefill

\begin{ex}
We give examples of each case on the set $[9]$ when considering the initial block:
\begin{description}
\item[Case A:] With $A=\{1,3\}$ and $B=\{2\}$, the partition $P= \cdots /47/8$ would do a merge move producing $i_9(P) = \cdots /478$.

\item[Case B:] With $A=\{1,3\}$ and $B=\{2\}$, the partition $P=\cdots /48/7$ would do a block skip on $7$ and will move to consider block $48$. Therefore $i_9(P) = \cdots /4/8/7$.  

\item[Case C:] With $A=\{1,3\}$ and $B=\{2\}$, the partition $P=\cdots/28/467/9$ would do a freeze skip, and next consider the block $28$. So $i_9(P) = \cdots /2/8/467/9$. 

Note also here that $P_1=\cdots/458/7$ would do a block skip and not a freeze skip, as $7$ is smaller than $8$. In this case $i_9(P_1) = \cdots/45/8/7$. 

\item[Case D:] With $A=\{1,3,5\}$ and $B=\{4\}$, the partition $P=\cdots/3/2689$ would make a split move producing $i_9(P) = \cdots/3/268/9$.

\item[Case E:] With $A=\{1,3,5\}$ and $B=\{4\}$, the partition $P=\cdots/4/269$ would do a block skip on $269$ (note $2 \notin B$) and next consider the singleton block $4$, and use induction on $\cdots/4$.
\end{description}
\end{ex}

\begin{remark}
The motivation for the freeze skip above is to make $i_n$ an involution.  To illustrate this, suppose $A=\{1,3\}$, $B=\{2\}$, the partition $P=134/28/567/9$, and we did not perform a freeze skip but instead a block skip on the first block.  Then $i_9(P) = 134/28/56/7/9$ (as the block containing $9$ is skipped, and a split move is performed on $567$), and therefore $i_9(i_9(P)) = 134/28/56/79 \neq P$.
\end{remark}

We now prove that $i_n$ has the properties required. A few of them are immediate. First, the map $i_n$ is indeed a map from ${\mathcal P}^{\rm pos}_{n, A,B} \cup {\mathcal P}^{\rm neg}_{n, A,B}$ to itself, as all block sizes of $i_n(P)$ are, by construction, in $A\cup B$.   Second, it preserves underlying permutations, as only merges and splits of blocks occur. 

We next show that $i_n$ is an involution, and we do this by induction on $n$ with base cases $n=0, 1$ trivial. For $n \geq 2$, we now show that whichever of Cases A, B, C, D or E the ordered partition $P=P_1/\cdots/P_\ell$ falls into, we have  $i_n(i_n(P))=P$. We will use the notation $P':=i_n(P)$ in each case.
\begin{description}
\item[Case A] Here $P'=P_1/P_2/\cdots/P_{\ell-1}P_\ell$. The last block of $P'$ has size greater than $1$, and $|P_{\ell-1}P_\ell|-1 = |P_{\ell-1}| \in A\cup B$. So $P'$ falls into Case D, and so $i_n(P')=P$. 
\item[Case B] Here $P'=i_{n-1}(P_1/P_2/\cdots/P_{\ell-1})/P_\ell$. The last block in this ordered partition is a singleton, say with entry $k$. This number $k$ is smaller than the right-most entry of $P_{\ell-1}$, so 
$P'$ falls into Case B (as the underlying permutation for $P$ and $P'$ is the same), and
$$i_n(P')=i_{n-1}(i_{n-1}(P_1/P_2/\cdots/P_{\ell-1}))/P_\ell.$$ 
Since $i_{n-1}$ is an involution (by induction), $i_n(P')=P$.
\item[Case C] Here $P'=i_{n-1-|P_{\ell-1}|}(P_1/P_2/\cdots/P_{\ell-2})/P_{\ell-1}/P_\ell$, $P_\ell$ is a singleton block with entry $k$ and $k$ is larger than the largest entry in $P_{\ell-1}$, and also $|P_{\ell-1}|+1 \notin A \cup B$. So $P'$ falls into Case C, and therefore  
$$i_n(P')=i_{n-1-|P_{\ell-1}|}(i_{n-1-|P_{\ell-1}|}(P_1/P_2/\cdots/P_{\ell-2}))/P_{\ell-1}/P_\ell.$$
Since $i_{n-1-|P_{\ell-1}|}$ is an involution (by induction), $i_n(P')=P$.
\item[Case D] Here $P'=P_1/P_2/\cdots/P_\ell-k/k$. The last block of $P'$ has size $1$, and the entry in that block, $k$, is greater than all entries in the penultimate block of $P'$. 
Also, the size of the penultimate block of $P'$ is $|P_\ell|-1$ and $(|P_{\ell}|-1)+1 \in A \cup B$. So $P'$ falls into Case A, and $i_n(P')=P$.
\item[Case E] Here $P'=i_{n-|P_\ell|}(P_1/P_2/\cdots/P_{\ell-1})/P_\ell$. The last block of $P'$ has size greater than $1$ and $|P_{\ell}|-1 \notin A \cup B$.  So $P'$ falls into Case E, and 
$$
i_n(P')=i_{n-|P_\ell|}(i_{n-|P_\ell|}(P_1/P_2/\cdots/P_{\ell-1}))/P_\ell.
$$
Since $i_{n-|P_\ell|}$ is an involution (by induction), $i_n(P')=P$.
\end{description}

\noindent We have thus shown that $i_n$ is an involution. 

\bigskip

To show that $i_n$ is sign-reversing, we will show that if $i_n(P) \neq P$, then $i_n(P)$ and $P$ differ by one in the number of blocks with even size.  But this follows from the construction: either blocks of size $t-1$ and $1$ merge to a single block of size $t$, or a block of size $t>1$ splits to two with size $t-1$ and $1$. So if $i_n(P) \neq P$, then exactly one of $i_n(P)$ and $P$ is in ${\mathcal P}^{\rm pos}_{n,A,B}$.

Finally we argue that for $n \geq 1$, if any block size of $P$ is in $B$, then $i_n(P)\neq P$. As a consequence, the only fixed ordered permutations have block sizes entirely in $A$.  In fact, we show something stronger, which will lead to our combinatorial interpretations. 
We note that being odd-ended implies that all top and bottom elements are in $A$. 
\begin{lemma}\label{lem-topbottomelts}
Let $n \geq 1$. With the involution $i_n$ defined above, if $i_n(P) = P$, then all block sizes of $P$ are $1$ or a top element or a bottom element that is not equal to $1$. 

Furthermore, the necessary and sufficient conditions on the order of the block sizes of $P$ are that, (i), the top elements that are not also bottom elements must be part of a freeze skip pair, and, (ii), the bottom elements that are not also top elements must not have a singleton to the right (whose elements are collectively in increasing order) that is not in a freeze skip pair.
\end{lemma}

We note that $1$ is a bottom element, but we separate it out in Lemma \ref{lem-topbottomelts} to emphasize its importance in the algorithm.  To better understand condition (ii), suppose that a bottom element that is not a top element has a singleton to the right whose elements are in increasing order.  If that singleton is in a freeze-skip pair, then it must be the block that is ``frozen'' and so has a second singleton to its right (whose elements are collectively in increasing order), and therefore we must have $2 \notin A \cup B$.  So, if $2 \in A \cup B$, then (ii) simply says that if you have a bottom element that is not a top element, then it does not have a singleton to the right (whose elements are collectively in increasing order).  Whereas if $2 \notin A \cup B$, then the second condition inductively says that if you have a bottom element that is not a top element, then it does not have an odd number of singletons to the right (whose elements are collectively all in increasing order).

\begin{proof}[Proof of Lemma \ref{lem-topbottomelts}]
To prove the first statement, we proceed by induction on $n$, with the result trivial for base case $n=1$. Assume that the ordered partition $P=P_1/\cdots/P_\ell$ has at least one block size not equal to $1$ or a top element or a bottom element.  We argue that $i_n(P) \neq P$.  This is obvious in Cases A and D as here we immediately have $i_n(P)\neq P$.  In Case B, we skip a block with size $1$, in Case C we skip blocks of size $1$ and a top element, and in Case E we skip a block with size being a bottom element.  In each of these three Cases, we are left with a partition that still has at least one block size not equal to $1$ or a top element or a bottom element, and so by induction we obtain $i_n(P) \neq P$.

The second statement is clear from the algorithm, as it is describing the properties needed so that Cases A and D never occur when considering a fixed block, and so $P$ is fixed.  Note that condition (ii) ensures that case A can't happen and (i) ensures that case D can't happen.
\end{proof}

\begin{ex}
Suppose that $A = \{1,3,5\}$ and $B=\{4\}$, so $A \cup B = \{1,3,4,5\}$.  Then the partition $1/234/5/6$ of $[6]$ is a fixed point of $i_6$.  However, the partition $123/4/5/6$ is not a fixed point of $i_6$, since $i_6(123/4/5/6) = 1234/5/6$.

This illustrates that partitions that merely satisfy that the block sizes are $1$ or a top element or a bottom element need not be fixed by the involution.
\end{ex}

\medskip

This completes the proof of the existence of a sign-reversing involution $i_n$ on ${\mathcal P}^{\rm pos}_{n,A,B} \cup {\mathcal P}^{\rm neg}_{n,A,B}$ whose fixed points all lie in ${\mathcal P}^{\rm pos}_{n,A,B}$. In fact, we have proved that all of the fixed points are elements of ${\mathcal P}^{\rm pos}_{n,A,B}$ with all block sizes $1$ or top elements or bottom elements and that the orderings of these blocks must satisfy certain conditions.

\subsection{General Combinatorial Interpretations}

With the terminology from the proof and the characterization of the fixed points of $i_n$, we have the following as a Corollary of the proof of Theorem \ref{thm-odd-ended} and Lemma \ref{lem-topbottomelts}. 

\begin{defn}\label{def-i_nfixed}
Suppose that $A \subseteq \mathcal{O}$ and $B \subseteq \mathcal{E}$ with $1 \in A$. Say that an ordered set partition of $[n]$ is {$A,B$-good} if:
\begin{itemize}
    \item all parts are either size $1$ or a size that is an endpoint of the maximal intervals of $A \cup B$,
    \item writing the elements of a block in increasing order, within an increasing run in the underlying permutation the following holds:
    \begin{itemize}
        \item all blocks with size equal to a top element that is not a bottom element is part of a freeze skip pair, and
        \item all blocks with size equal to a bottom element that are not also top elements do not have a singleton immediately to the right that is not in a freeze skip pair.
    \end{itemize}
\end{itemize}
\end{defn}

In light of Lemma \ref{lem-topbottomelts}, when $A \cup B$ is odd-ended, the ordered set partitions of $[n]$ that are $A,B$-good are exactly those which the involution $i_n$ leaves fixed.

\begin{cor}\label{cor-partitioninterp}
Suppose that $A \subseteq \mathcal{O}$ and $B \subseteq \mathcal{E}$ with $A \cup B$ odd-ended and $1 \in A$.  
Define $(c_n)_{n \geq 0}$ via 
\[
\left(1-\sum_{a \in A} \frac{x^a}{a!} + \sum_{b \in B} \frac{x^b}{b!} \right)^{-1} = \sum_{n \geq 0} c_n \frac{x^n}{n!}.
\]
Then the values $c_n$ are all non-negative integers, and moreover have a natural combinatorial interpretation as the number of $A,B$-good ordered set partitions of $[n]$.
\end{cor}

Instead of looking at ordered set partitions, we can also imagine the underlying permutation being fixed, with a particular permutation being counted multiple times based on the number of ways of partitioning the maximal increasing runs into allowable blocks with an $A,B$-good ordered set partition.

\begin{ex}
Let $A \cup B = \{1,3\}$. We illustrate the difference between fixing a permutation and fixing an ordered set partition in our combinatorial interpretation.
\begin{enumerate}
    \item[(a)] Fixing a permutation and counting the ordered set partitions that could have that permutation underlying them: 
    
    Suppose we consider the permutation 456123 of $[6]$. This permutation is counted in four ordered set partitions: as $456/123$, $4/5/6/123$, $456/1/2/3$, or $4/5/6/1/2/3$.

    \item[(b)]Fixing the ordered set partition and counting the underlying permutations that give that set partition:
    
    Suppose that there are two blocks with size $3$ in an ordered set partition of $[6]$. By choosing the three elements to be in the first block, there are $\binom{6}{3} = 20$  permutations that can have two blocks with size $3$. Furthermore, there are $4 \binom{6}{3}\cdot 3! = 120$ that have one block of size three and three singleton blocks.  Indeed, there are 4 places to start the block of size 3, and the binomial coefficient and factorial count the number of ways to distribute 3 elements to the block of size three and the singleton blocks.
    Lastly, there are $6!$ composed of all singleton blocks.  Note that for $A \cup B = \{1,3\}$, we indeed have $c_6=1220$, which is computed as $\binom{6}{3} + 4 \cdot \binom{6}{3} \cdot 3! + 6!$.
\end{enumerate}
\end{ex}

When viewing the ordered set partitions that are fixed by $i_n$ as grouped together based on their underlying permutation, we arrive at the following definition. 

\begin{defn}
Suppose that $A \subseteq \mathcal{O}$ and $B \subseteq \mathcal{E}$ with $1 \in A$. Let $\ell \in [n]$.  The {\em weight $w_\ell$} is defined as the number of ordered set partitions of $\ell$ ordered elements so that:
\begin{itemize}
    \item all blocks are either size $1$ or a size that is an endpoint of the maximal intervals,
    \item all blocks with size equal to a top element that is not a bottom element is part of a freeze skip pair, and
    \item all blocks with size equal to a bottom element that is not a top element do not have a singleton immediately to the right that is not in a freeze skip pair.
\end{itemize}
\end{defn}

\begin{defn}
For a permutation $\sigma$ of $[n]$, define the {\em weight of $\sigma$}, denoted $w_\sigma$, as the product of the weights of the lengths of all maximal increasing runs of $\sigma$.
\end{defn}

\begin{lemma} \label{lem-w_sigma}
Given a permutation $\sigma$ of $[n]$, the number of $A,B$-good partitions that could have $\sigma$ underlying them is $w_\sigma$.
\end{lemma}
\begin{proof}
    Suppose first that $\sigma$ is a run, i.e. $\sigma=12 \ldots n$.  By Definition \ref{def-i_nfixed}, $w_n$ is the number of $A,B$-good partitions that could have $\sigma$ underlying them. Now suppose that $\sigma$ has runs of lengths $l_1, l_2, \ldots, l_k$. Consider an $A,B$-good partition with $\sigma$ underlying it. Since a block must have its elements in increasing order, each block must lie entirely in some run. Similarly the elements in both blocks of a freeze-skip pair must come in increasing order and thus also lie entirely in some run.  Thus the number of $A,B$-good partitions that could have $\sigma$ underlying them is $w_{l_1} w_{l_2} \cdots w_{l_k} = w_\sigma$.
\end{proof}
\begin{ex}
If we let $A \cup B = \{1,3\}$, then the permutation $\sigma = 4567123$ has a maximal increasing run of length $3$ that satisfies $w_3=2$ and a maximal increasing run of length $4$ that satisfies $w_4$=3, and so $w_\sigma=2\cdot 3 = 6$.
\end{ex}

\begin{cor}\label{cor-perminterp}
Suppose that $A \subseteq \mathcal{O}$ and $B \subseteq \mathcal{E}$ with $A \cup B$ odd-ended and $1 \in A$.  
Define $(c_n)_{n \geq 0}$ via 
\[
\left(1-\sum_{a \in A} \frac{x^a}{a!} + \sum_{b \in B} \frac{x^b}{b!} \right)^{-1} = \sum_{n \geq 0} c_n \frac{x^n}{n!}.
\]
Then the values $c_n$ are all non-negative integers, and moreover have a natural combinatorial interpretation as  
\[
c_n = \sum_\sigma w_\sigma,
\]
where the sum is over all permutations $\sigma$ of $[n]$.

\end{cor}
\begin{proof}
By Corollary \ref{cor-partitioninterp}, $c_n$ is the number of $A,B$-good set partitions.  By Lemma \ref{lem-w_sigma}, if $\sigma$ is a permutation of $[n]$, $w_\sigma$ is the number of $A,B$-good partitions that could have the permutation $\sigma$ underlying them.  
\end{proof}

See Section \ref{sec-run-theorem} for alternate combinatorial interpretations that are equivalent to the results in Corollaries \ref{cor-partitioninterp} and \ref{cor-perminterp} that follow from the Run Theorem.

\subsection{Recovering Gessel's Interpretation Involving Increasing Runs}
In the special case $A \cup B = [2m-1]$ for some positive integer $m$, we now show how we can recover Gessel's interpretation that $c_n$ is the number of permutations of $[n]$ whose maximal increasing runs all have length congruent to $0$ or $1$ mod $2m$.  

When $m=1$, the reciprocal is $(1-x)^{-1} = \sum_{n=0}^\infty x^n$, which has exponential generating function coefficient $n!$.  Here Gessel's characterization is trivial as every integer is congruent to $0$ or $1$ mod $2$, and so all $n!$ permutations are counted and $c_n=n!$.

For $m>1$, the element $1$ is a bottom element that is not a top element, and the element $2m-1$ is a top element that is not a bottom element.  Therefore, by Corollary \ref{cor-perminterp} and Definition \ref{def-i_nfixed}, an increasing run is made up from blocks with sizes $1$ and $2m-1$; further, the blocks of size $2m-1$ must have a block of size $1$ to the right, and there are no consecutive blocks of size $1$.

It follows that the the maximal increasing runs in the permutation must have the blocks of the ordered set partition appear in blocks that occur in freeze skip pairs, with the extra possibility of one more singleton block on the left of these pairs.  Since freeze skip pairs have $2m$ elements in them, if there are $j$ freeze skip pairs there are either $2mj$ elements in the increasing run (if there is no singleton block to the left) 
or $2mj+1$ elements in the increasing run (if there is a singleton block to the left). 
This shows that all increasing runs must have length $0$ or $1$ mod $2m$.
Finally, every permutation whose increasing runs have length $0$ or $1$ mod $2m$ has a unique associated ordered set partition $P$ where the runs consist of freeze skip pairs (with a possible singleton block on the left). Therefore the weight of each such maximal increasing run length is $1$, and so the weight of those permutations where all run lengths are $0$ or $1$ mod $2m$ is $1$ (with any other permutation having weight $0$).

\begin{ex}
With $A=\{1,3\}$ and $B=\{2\}$ we have $m=2$.  We list the some permutations on $9$ elements with maximal increasing runs of size $0$ or $1$ mod $4$, along with the corresponding partitions fixed by $i_9$ (with increasing runs underlined).

\[
\begin{array}{c|c|c}
\text{Permutation} & \text{Permutation with Runs Marked} & \text{Partition}\\
\hline
&&\\
987654321 & \underline{9}\,|\,\underline{8}\,|\,\underline{7}\,|\,\underline{6}\,|\,\underline{5}\,|\,\underline{4}\,|\,\underline{3}\,|\,\underline{2}\,|\,\underline{1} & \underline{9_{_{}}}/\underline{8_{_{}}}/\underline{7_{_{}}}/\underline{6_{_{}}}/\underline{5_{_{}}}/\underline{4_{_{}}}/\underline{3_{_{}}}/\underline{2_{_{}}}/\underline{1_{_{}}}\\
&&\\
234615789 & \underline{2346}\,|\,\underline{15789} & \underline{234/6}/\underline{1/578/9}\\
&&\\
721689543 & \underline{7}\,|\,\underline{2}\,|\,\underline{1689}\,|\,\underline{5}\,|\,\underline{4}\,|\,\underline{3} & \underline{7_{_{}}}/\underline{2_{_{}}}/\underline{168/9}/\underline{5_{_{}}}/\underline{4_{_{}}}/\underline{3_{_{}}}\\
&&\\
712689543 & \underline{7}\,|\,\underline{12689}\,|\,\underline{5}\,|\,\underline{4}\,|\,\underline{3} & \underline{7_{_{}}}/\underline{1/268/9}/\underline{5_{_{}}}/\underline{4_{_{}}}/\underline{3_{_{}}}\\
&&\\
\end{array}
\]
\end{ex}

Thus, our combinatorial interpretation for the special case when $A\cup B=[2m-1]$ is exactly as in \cite{Gessel}.

When $A \cup B = \{1,2,3,\ldots\}$, the reciprocal is $e^x$ which has exponential generating function coefficients all equal to $1$.  It is immediate from Corollary \ref{cor-perminterp} and Definition \ref{def-i_nfixed} that, since $1$ is the only endpoint of a maximal interval, that the coefficients $c_n$ count the decreasing permutation that uses only singleton blocks, and so $c_n=1$. 

These cover the examples 
containing a $1$ and only a top element, and those containing only a $1$ and no other top or bottom elements.  

\begin{ex}

What if we had $1$ and only a second bottom element, such as $A \cup B = \{1, 5,6,7,8,\ldots\}$? Then the maximal increasing runs have blocks with size $1$ and $5$ only.  Furthermore, no size $5$ block can be have a block with size $1$ to the right of it (with elements in increasing order). We can, however, have arbitrarily many blocks with size $1$ in a row in increasing order. Therefore we have $w_\ell = \lfloor \ell/5 \rfloor + 1$ as there are this many ways to choose the number of blocks of size $5$ to use on the largest elements of the run, with the rest of the blocks singletons. 
\end{ex}

\section{Proof of Theorem \ref{thm-odd-endedstretched} and Combinatorial Interpretations}\label{sec-stretched}
In this section we prove Theorem \ref{thm-odd-endedstretched} by considering $b>1$ and $r>2$.  Following this, we give the combinatorial interpretations of the results in this setting.

\subsection{Stretching by $b$}
To obtain results for $b > 1$, we simply blow-up each element in $[n]$ to a group of $b$ distinct increasing elements and run the Case $r=2$ Algorithm on the groups. In particular, any moves done on the ordered permutation of $[n]$ translate to moving the corresponding group of $b$ blow-up elements in the natural way.

One tangible way to do this is to consider the elements of $[n]$ as fractions $\left\{\frac{1}{1},\frac{2}{1},\ldots,\frac{n}{1}\right\}$, and we then replace each element $\frac{i}{1}$ with the $b$ distinct elements $\frac{bi-b+1}{b}, \frac{bi-b+2}{b},\ldots,\frac{bi}{b}$ (corresponding to the fractions $x$ with denominator $b$ such that $\frac{b(i-1)}{b} < x \leq \frac{bi}{b}$). 

\begin{ex}
If $b=3$, $A = \{1,3\}$, $B= \{2,4\}$, and we had a partition $P=13/24$ of $[4]$, recall that $i_4(P) = 13/2/4$. In the blow-up by $b$, we would replace this partition $P$ with the partition
$$\frac{1}{3}\,\,\frac{2}{3}\,\,\frac{3}{3}\,\,\frac{7}{3}\,\,\frac{8}{3}\,\,\frac{9}{3} \quad / \quad \frac{4}{3}\,\,\frac{5}{3}\,\,\frac{6}{3}\,\,\frac{10}{3}\,\,\frac{11}{3}\,\,\frac{12}{3},$$ 
and the algorithm performs a split move on the furthest right block, which now is a blow-up to two groups each of size $3$, to create 
$$\frac{1}{3}\,\,\frac{2}{3}\,\,\frac{3}{3}\,\,\frac{7}{3}\,\,\frac{8}{3}\,\,\frac{9}{3} \quad / \quad \frac{4}{3}\,\,\frac{5}{3}\,\,\frac{6}{3}\quad / \quad \frac{10}{3}\,\,\frac{11}{3}\,\,\frac{12}{3},$$ 
By ignoring the denominators, we can also view a blow-up of $[n]$ to $[bn]$ where $i \in [n]$ corresponds to the $b$ elements $b(i-1)+j$ for $1 \leq j \leq b$. 

\end{ex}

The proof generalizes easily with this modification, as each permutation element in the original proof simply corresponds to $b$ elements in this extension.
Combinatorial interpretations when $b>1$ simply scale every element to a group with $b$ elements in a straightforward manner.

\subsection{Stretching by $r$}

We now describe how to obtain results for $r >2$. We will assume that $b=1$; using $b>1$ will require the same modification given in the previous section.

As mentioned in the introduction, instead of considering $A \cup B \subseteq \NN$ we now consider $A \cup B \subseteq \NN^* = \{1,r,r+1,2r,2r+1,\ldots\} = \{ x \in \NN : x \equiv 0,1 \mod r\}$.  We write $\NN^* = {\cal O}^* \cup {\cal E}^*$ where ${\cal O}^* = \{x \in \NN : x \equiv 1\mod r\}$ is the set of $r$-stretched odds and ${\cal E}^* = \{x \in \NN : x \equiv 0\mod r\}$ is the set of $r$-stretched evens.  We let $A \subseteq {\cal O}^*$ and $B \subseteq {\cal E}^*$.  Note also that this situation is not covered by Theorem \ref{thm-odd-ended}, since for example if $r=3$ then the number $4$ would appear in $B$, not $A$.  

\medskip

\noindent {\bf Note:} Throughout we will naturally generalize our definitions and assumptions to this stretched setting.  We will always assume $1 \in A$.  We will also assume that $A \cup B$ is odd-ended, i.e. its maximal intervals in $\NN^*$ have endpoints that all lie in ${\cal O}^*$.  In other words, $A \cup B$ must satisfy that if $kr \in B$, then $kr+1 \in A$ and $(k-1)r+1 \in A$. 

We also then generalize the definition of top (bottom) element to indicate that $x \in A \cup B$ and the element above (below) it in $\NN^*$ is \emph{not} in $A \cup B$.  For example, $kr+1 \in A$ is a top element if $(k+1)r$ is not an element of $A \cup B$.

\begin{ex}
    Suppose that $r=6$.  Then ${\cal O}^* = \{1,7,13,19,\ldots\}$ and ${\cal E}^* = \{6,12,18,24,\ldots\}$.  The set $A \cup B = \{1,7,12,13\} \subseteq \mathbb{N}^*$ is odd-ended in this setting with bottom elements $1$ and $7$ and top elements $1$ and $13$.
\end{ex}

Given $A \cup B$ odd-ended, for each non-negative integer $n$ we again inductively construct an involution $i_n$ on $\mathcal{P}^{\text{pos}}_{n,A,B} \cup \mathcal{P}^{\text{neg}}_{n,A,B}$ that is sign-reversing and moves every element in $\mathcal{P}^{\text{neg}}_{n,A,B}$ (or, equivalently, only fixes partitions in $\mathcal{P}^{\text{pos}}_{n,A,B}$).  By saying $i_n$ is sign-reversing, we mean that if $i_n(P) \neq P$ then $i_n(P) \in \mathcal{P}^{\text{pos}}_{n,A,B}$ if $P \in \mathcal{P}^{\text{neg}}_{n,A,B}$ and $i_n(P) \in \mathcal{P}^{\text{neg}}_{n,A,B}$ if $P \in \mathcal{P}^{\text{pos}}_{n,A,B}$.  As before, we will show that the underlying permutation on $i_n(P)$ is the same as the underlying permutation on $P$.  The inductive procedure can equivalently be thought of as an iterative algorithm on a given ordered partition.

The main difference in this proof from the proof of Theorem \ref{thm-odd-ended} is that merges no longer always involve a singleton and another block, but instead involve either $r-1$ singletons and a block with size $kr+1$ or a singleton and a block of size $kr$ (so that all sizes are in $A \cup B$).  Having different types of merge moves means there are more cases to consider, but the structure of the algorithm is the same.  The base cases and inductive hypotheses here are identical with those of Theorem \ref{thm-odd-ended}.

\bigskip

\noindent \textbf{Constructing $i_0$:} When $n=0$ there is a unique ordered partition, the empty partition, which is in $\mathcal{P}^{\text{pos}}_{0,A,B}$, as it has zero blocks with sizes from $B$.  The involution is the identity map, and the relative order condition is trivially satisfied.

\bigskip

\noindent \textbf{Constructing $i_1$:} When $n=1$ there is again a unique ordered partition, with one singleton block $\{1\}$, which is in $\mathcal{P}^{\text{pos}}_{1,A,B}$, as it has zero blocks with sizes from $B$.  The involution here is again the identity map, and also the relative order condition is trivially satisfied.

\bigskip

\noindent {\bf Constructing $i_n$ for $n \geq 2$}: We take as our inductive hypothesis that for $m < n$, a sign-reversing involution $i_m$ on ${\mathcal P}^{\rm pos}_{m, A,B} \cup {\mathcal P}^{\rm neg}_{m, A,B}$, all of whose fixed points lie in ${\mathcal P}^{\rm pos}_{m,A,B}$, has been constructed with the property that for every $P \in {\mathcal P}^{\rm pos}_{m, A,B} \cup {\mathcal P}^{\rm neg}_{m,A,B}$, the permutation underlying $i_n(P)$ is the same as that underlying $P$.

We first construct an explicit involution $i_n$  on ${\mathcal P}^{\rm pos}_{n, A,B} \cup {\mathcal P}^{\rm neg}_{n, A,B}$, that also satisfies the relative order condition.

Suppose that $P:=P_1/P_2/\cdots/P_\ell$ is an ordered partition of $[n]$ with $|P_i| \in A \cup B$ for $i=1,\ldots,\ell$, where the elements inside $P_i$ are listed in increasing order.

\noindent \hrulefill

\noindent \textbf{ALGORITHM: Case $r >2$}

\begin{description}
\item[(1)] Suppose that $|P_\ell|=1$, say block $P_\ell$ consists of the single number $k_{r-1}$.
\begin{description}
\item[(A)] Suppose $|P_{\ell-1}|=kr$.

\begin{description}
\item[Case A1] If $k_{r-1}$ is smaller than the largest entry in $P_{\ell-1}$, then do a {\em block skip} (we will also call a block skip that skips a singleton a {\em singleton skip}): 
fix $P_\ell$ and next consider $P_{\ell-1}$, so $i_n(P) = i_{n-1}(P_1/P_2/\cdots/P_{\ell-1})/P_\ell$. 

\item[Case A2] If $k_{r-1}$ is larger than the largest entry of $P_{\ell-1}$, make a {\em 1-merge move}\footnote{Note that $|P_{\ell-1}|+1 \in A \cup B$ as we are odd ended}: combine blocks $P_{\ell-1}$ and $k_{r-1}$ to produce a block with size $kr+1$, so $i_n(P) = P_1/P_2/\cdots/P_{\ell-1}P_\ell$. 
\end{description}

\item[(B)] Suppose $|P_{\ell-1}|=kr+1$.
\begin{description}

\item[Case B1] If $k_{r-1}$ is to the right of $r-2$ singletons $k_1,\ldots,k_{r-2}$ followed by $P_{\ell-(r-1)}$ with size $kr+1$, and it is not true that the elements of $P_{\ell-(r-1)}$, $k_1$, $\ldots$, $k_{r-1}$ are in increasing order: $k_1$ is larger than the largest element of $P_{\ell-(r-1)}$ and $k_1<k_2<\cdots<k_{r-1}$, then do a \emph{singleton skip}: fix $P_\ell$ and next consider $P_{\ell-1}$, so $i_n(P) = i_{n-1}(P_1/P_2/\cdots/P_{\ell-1})/P_\ell$.

\item[Case B2] If $k_{r-1}$ is to the right of $r-2$ singletons $k_1,\ldots,k_{r-2}$ followed by $P_{\ell-(r-1)}$ with size $kr+1$, the elements of $P_{\ell-(r-1)}$, $k_1$, $\ldots$, $k_{r-1}$ are in increasing order,  
and $P_{\ell-(r-1)}$ has size $kr+1$ which is not a top element of $A \cup B$, then make a {\em $(r-1)$-merge move}: combine blocks $P_{\ell-(r-1)},k_1,\ldots,k_{r-1}$, so $i_n(P)=P_1/P_2/\cdots/P_{\ell-(r-1)}k_1\cdots k_{r-1}$.

\item[Case B3] If $k_{r-1}$ is to the right of $r-2$ singletons $k_1,\ldots,k_{r-2}$ followed by $P_{\ell-(r-1)}$ with size $kr+1$, the elements of $P_{\ell-(r-1)}$, $k_1$, $\ldots$, $k_{r-1}$ are in increasing order,  
and $P_{\ell-(r-1)}$ has size $kr+1$ which {\em is} a top element of $A \cup B$, 
then make a {\em $r$-freeze skip}: remove $P_{\ell-(r-1)},k_1,\ldots,k_{r-1}$ from the partition, then proceed inductively, so we have

$i_n(P)=i_{n-r+1-|P_{\ell-(r-1)}|}(P_1/P_2/\cdots/P_{\ell-r})/P_{\ell-(r-1)}/k_1/\cdots/k_{r-1}$.

\item[Case B4] If $k_{r-1}$ is to the right of $r-2$ singletons $k_1,\ldots,k_{r-2}$ followed by $P_{\ell-(r-1)}$ with size $kr$, then do a {\em singleton skip}. 

\item[Case B5] If $k_{r-1}$ is to the right of fewer than $r-2$ singletons followed by a non-singleton, do a {\em singleton skip}.
\end{description}

\end{description}
\end{description}

\begin{description}
\item[(2)] Suppose that $|P_\ell|>1$, with $k_1,\ldots,k_{r-1}$ the largest $r-1$ element of $P_\ell$.
\begin{description}
\item[(C)] Suppose $|P_\ell|$ is {\em not} a bottom element of $R$.
\begin{description}
\item[Case C1] If $|P_{\ell}|=kr$, then make a {\em $(r-1)$-split move}: replace the block $P_\ell$ with $r$ blocks, $P_\ell-k_1-\cdots-k_{r-1}$, $k_1$, $k_2$, $\ldots$, $k_{r-1}$, in that order, so $i_n(P)=P_1/P_2/\cdots/P_{\ell-1}/P_\ell-k_1-\cdots-k_{r-1}/k_1/\cdots/k_{r-1}$. 

\item[Case C2] If $|P_{\ell}|=kr+1$, then make a {\em $1$-split move}: replace the block $P_\ell$ with two blocks, $P_\ell-k_{r-1}$ and $k_{r-1}$, in that order, so $i_n(P)=P_1/P_2/\cdots/P_{\ell-1}/P_\ell-k_{r-1}/k_{r-1}$. 

\end{description} 
\item[Case D]  If $|P_\ell|$ {\em is} a bottom element of $R$, then do a {\em block skip}. 
\end{description}
\end{description}

\noindent \hrulefill

\bigskip

As in the previous algorithm, few things are immediate. First, the map $i_n$ is indeed a map from ${\mathcal P}^{\rm pos}_{n, A,B} \cup {\mathcal P}^{\rm neg}_{n, A,B}$ to itself, as all block sizes of $i_n(P)$ are, by construction, in $A\cup B$.   Second, it preserves underlying permutations, as only merges and splits of blocks occur.

We now check that it is an involution, proceeding again by induction on $n$, with base cases $n=0,1$ trivial.  We show that whichever case the partition $P=P_1/\cdots/P_{\ell}$ falls into, $i_n(i_n(P))=P$. Throughout let $P':=i_n(P)$.

\begin{description}
\item[Case A1] Here $P' = i_{n-1}(P_1/\cdots/P_{\ell-1})/P_{\ell}$. 
Since the order of the underlying permutation is preserved, the partition $P'$ falls into Case A1, so
\[
i_n(P') = i_{n-1}(i_{n-1}(P_1/P_2/\cdots/P_{\ell-1}))/P_{\ell}.
\]
Since $i_{n-1}$ is an involution (by induction), $i_n(P')=P$. 
\item[Case A2] Here $P'$ falls into Case C2, and evidently $i_n(P')=P$.
\item[Case B1] Our goal here is that $P'$ again must do a singleton skip, as then the result will hold by induction; indeed, then, 
\[
i_n(P') = i_{n-1}(i_{n-1}(P_1/P_2/\cdots/P_{\ell-1}))/P_\ell.
\]
Recall that the order of the underlying permutation does not change. First, suppose that $k_{r-2}>k_{r-1}$. Then the partition $P'$ falls into Case A1, B1, B4, or B5 (as $P_\ell$ is a singleton, and Cases A2, B2, and B3 require $k_{r-2}<k_{r-1}$) and so does a singleton skip again. 

If instead we have $k_{r-2}<k_{r-1}$, then as $P$ is in Case B1 the $r$ elements given by the largest element of $P_{\ell-(r-1)}$, $k_1$, $\ldots$, $k_{r-1}$ are not all in increasing order, so cannot all be part of a single block of $P'$.  Therefore either $P'$ is in Case B5 (if some of the elements are in the same block in $P'$) or Case B1 or B5 (if none of those elements are in the same block in $P'$). In all cases $P'$ also does a singleton skip, so the result holds by induction.

\item[Case B2] Here $P'$ falls into Case C1, and evidently $i_n(P')=P$.
\item[Case B3] Here $P'$ falls into Case B3 again, and by induction $i_n(P')=P$.
\item[Case B4] Our goal is to show that we will again do a singleton skip in $P'$. Here the algorithm on $P$ will proceed with B5 singleton skips until reaching the last singleton $k_1$, when it will either fall into Case A1 and continue on (if the singleton $k_1$ is not larger than the largest entry in $P_{\ell-(r-1)}$), or it will fall into Case A2 (otherwise).  

In the former case, note first that $P$ will fall into Case C1 when arriving at $P_{\ell-(r-1)}$ and produce $r-1$ more singletons.  Therefore the partition $P'$ falls into Case B1 (as it will end in $r$ singletons that are not in order).  

In the latter case, the partition $P'$ has one fewer singleton to the left of $k_{r-1}$ and so $P'$ falls into Case B5. 

In all possibilities, $P'$ also does a singleton skip, and so induction implies that $i_n(P')=P$.

\item[Case B5] Again, our goal is to show that we will do a singleton skip in $P'$. The partition $P'$ will also be in Case B5 unless the algorithm produced more singleton elements in $P'$ to the left of $P_\ell$, which means that $P$ performed skips until reaching the first non-singleton to the left of $k_{r-1}$ and then performed either a 1-split (Case C2) or an $(r-1)$-split (Case C1).

If $P$ fell into Case C2, then in particular the non-singleton in $P$ is an odd block.  Therefore in $P'$ we have $r-1$ or fewer singletons followed by an even block.  If in $P'$ we have exactly $r-1$ singletons followed by an even block, the $P'$ falls into Case B4.  
If in $P'$ we have fewer than $r-1$ singletons followed by an even block, 
then $P'$ falls into Case B5.  

If instead $P$ fell into Case C1, then the non-singleton in $P$ is an even block, and in $P$ it did {\em not} perform Case A2 with the singleton to its right.  This means that the singleton to its right did a singleton skip in $P$ (Case A1) before reaching Case C1 in the non-singleton block, and so the entries of the non-singleton and the singleton to its right must be out-of-order in the underlying permutation.  Therefore the entries of $P'$ are also out of order, and so $P'$ falls into Case B1.  

In all cases the partition $P'$ will perform a singleton skip, and so induction implies that $i_n(P')=P$. 
\item[Case C1] Here $P'$ falls into Case B2, and evidently $i_n(P')=P$.
\item[Case C2] Here $P'$ falls into Case A2, and evidently $i_n(P')=P$.
\item[Case D] Here $P'$ falls into Case D again, and by induction $i_n(P')=P$.
\end{description}

In order for $i_n$ to be sign-reversing, we will show that if $i_n(P) \neq P$, then $i_n(P)$ and $P$ differ by one in the number of blocks with a size in $B$ (those whose size is equivalent to $0$ mod $r$). But this follows from the construction, since we either merge blocks of size $1$ and $kr$ to a single block of size $kr+1$, merge $r-1$ blocks of size $1$ and one of size $kr+1$ to a single block of size $(k+1)r$, or we reverse those moves via splits.  So if $i_n(P) \neq P$, then exactly one of $P$ and $i_n(P)$ is in $\mathcal{P}_{n,A,B}^{\text{pos}}$.

Finally, we show that if in $P$ any block has size equivalent to $0$ mod $r$ then $i_n(P) \neq P$. In fact, as before, we will that if $i_n(P)=P$, then all block sizes of $P$ are $1$ or a top element or a bottom element of $A \cup B$ with particular allowable orderings, which leads to our combinatorial interpretations.  Notice that as $A \cup B$ is odd ended, it follows that all top and bottom elements are in $A$.  Furthermore, all blocks with size $1$ are also in $A$.

\begin{lemma}\label{lem-topbottomeltsstretched}
Let $n \geq 1$.  With the involution $i_n$ defined above, if $i_n(P) = P$, then all block sizes of $P$ are $1$ or a top element or a bottom element.

Furthermore, the necessary and sufficient conditions on the order of the block sizes of $P$ are that, (i), the top elements that are not also bottom elements must be part of a freeze skip $r$-tuple, and\, (ii), the bottom elements that are not also top elements must not have $r-1$ singletons to the right (whose elements are collectively in increasing order) with none of those singleton blocks in a freeze skip $r$-tuple.
\end{lemma}

As we did following the statement of Lemma \ref{lem-topbottomelts}, we can understand condition (ii) based on whether $r$ is in $A \cup B$ or not.

So suppose that we take a bottom element that is not also a top element.  If any of the $r-1$ singletons to the right is part of an $r$-freeze skip, then it must be that $r$ consecutive singletons are frozen (whose elements are collectively in increasing order) and therefore $r \notin A \cup B$.  So if $r \in A \cup B$, then (ii) simply says that the bottom element that is not also a top element does not have $r-1$ singletons to the right (whose elements are collectively in increasing order).  But if $r \notin A \cup B$, then $r$ consecutive singletons can be ``frozen'' and so (ii) says that the bottom element that is not also a top element does not have exactly $jr + (r-1)$ singletons to the right (whose elements are collectively in increasing order), as this would perform $j$ consecutive $r$-freeze skips followed by an $(r-1)$-merge on the last $r-1$ elements with the bottom element that is not also a top element.  

To illustrate this, take for example $r=5$ and $A \cup B = \{1,6,10,11\}$ (and so $A = \{1,6,11\}$ and $B = \{10\}$, the block size $6$ is a bottom element that is not a top element, and $r=5 \notin A \cup B$).  Assuming all elements in the blocks are collectively in increasing order: 
\begin{itemize}
    \item  a block of size 6 with exactly $r-1=4$ singletons to the right will perform an $(r-1)$-merge to create a block with size 10 (Case B2);
    \item a block with size 6 with exactly $r+r=10$ singletons to the right will freeze the first $5$ singletons (Case B3), then freeze the next $5$ singletons (Case B3), then skip the block with size $6$ (Case D);  and
    \item a block with size $6$ with exactly $r+(r-1) = 9$ singletons to the right will freeze the first $5$ singletons (Case B3) and then perform an $(r-1)$-merge on the last $4$ singletons and the block with size $6$ (Case B2).
\end{itemize}
Note that if $i_n(P)=P$, then the first and third situations could not occur in $P$, while the second is possible in $P$.

\begin{proof}[Proof of Lemma \ref{lem-topbottomeltsstretched}]
We again proceed by induction on $n$, showing that if $P$ has at least one block size in $B$, then $i_n(P) \neq P$.  The base cases are trivial.  For $n \geq 2$ we consider an ordered set partition of $P=P_1/\cdots /P_\ell$ that has at least one block size in $B$. We consider the cases for $P$ under this assumption, showing why each lead to $i_n(P) \neq P$.

\begin{description}
\item[Case A1] By induction (we skip a block with size $1$).
\item[Case A2] Here $i_n(P)\neq P$.
\item[Case B1] By induction (we skip a block with size $1$).
\item[Case B2] Here $i_n(P)\neq P$.
\item[Case B3] By induction (we skip all blocks with sizes either $1 \in A$ or a top element of $A \cup B$ and so in $A$).
\item[Case B4] By induction (we skip a block with size $1$).
\item[Case B5] By induction (we skip a block with size $1$).
\item[Case C1] Here $i_n(P) \neq P$.
\item[Case C2] Here $i_n(P) \neq P$.
\item[Case D] By induction (we skip a size that is a bottom element of $A \cup B$ and is in $A$). 
\end{description}

The second statement is clear from the algorithm as well, as it is describing the properties needed so that Cases B2 and C2 do not occur; the other moves and splits cannot occur based on the block sizes all being in $A$. This finishes the proof.
\end{proof}

\subsection{General Stretched Combinatorial Interpretations}

The combinatorial interpretations now follow exactly as in Section \ref{sec-pfofodd-ended}.  We consider $b=1$, since the combinatorial interpretations for $b>1$ simply involve blowing up the elements to groups of size $b$.  These generalizations are completely analogous to those in Section \ref{sec-pfofodd-ended}, but we state them for completeness.

With the terminology from the proof and the characterization of the fixed points of $i_n$, we have the following natural definition. 

\begin{defn}\label{def-i_nfixedstretched}
Suppose that $A \subseteq \{x \geq 1: x \equiv 1 \mod r\}$ and $B \subseteq \{x \geq 1 : x \equiv 0 \mod r\}$. Say that an ordered set partition of $[n]$ is {$A,B,r$-good} if:
\begin{itemize}
    \item all blocks are either size $1$ or size that is an endpoint of the maximal intervals,
    \item writing the elements of a block in increasing order, within an increasing run in the underlying permutation the following holds:
    \begin{itemize}
        \item all blocks with size equal to a top element that is not a bottom element is part of a freeze skip $r$-tuple, and
        \item all blocks with size equal to a bottom element that are not also top elements do not have $r-1$ singletons immediately to the right with none of the singletons in a freeze skip $r$-tuple.
    \end{itemize}
\end{itemize}
\end{defn}

Generalizing Corollary \ref{cor-partitioninterp}, we have the following.

\begin{cor}\label{cor-partitioninterpstretched}
Suppose that $A \subseteq \{x \geq 1: x \equiv 1 \mod r\}$ and $B \subseteq \{x \geq 1 : x \equiv 0 \mod r\}$ with $A \cup B$ odd-ended and $1 \in A$.  
Define $(c_n)_{n \geq 0}$ via 
\[
\left(1-\sum_{a \in A} \frac{x^a}{a!} + \sum_{b \in B} \frac{x^b}{b!} \right)^{-1} = \sum_{n \geq 0} c_n \frac{x^n}{n!}.
\]
Then the values $c_n$ are all non-negative integers, and moreover have a natural combinatorial interpretation as the number of $A,B,r$-good ordered set partitions of $[n]$.
\end{cor}

We extend the definitions of weight to the $r$-stretched case as well.

\begin{defn}
Suppose that $A \subseteq \{x \geq 1: x \equiv 1 \mod r\}$ and $B \subseteq \{x \geq 1 : x \equiv 0 \mod r\}$ with $1 \in A$. Let $\ell \in [n]$.  The {\em weight $w_\ell$} is defined as the number of ordered set partitions of $\ell$ ordered elements so that:
\begin{itemize}
    \item all blocks are either size $1$ or size that is an endpoint of the maximal intervals,
    \item all blocks with size equal to a top element that is not a bottom element is part of a freeze skip $r$-tuple, and
    \item all blocks with size equal to a bottom element that are not also top elements do not have $r-1$ singletons immediately to the right with none of the singletons in a freeze skip $r$-tuple.
\end{itemize}
\end{defn}

\begin{defn}
For a permutation $\sigma$ of $[n]$, define the {\em weight of $\sigma$}, denoted $w_\sigma$, as the product of the weights of the lengths of all maximal increasing runs of $\sigma$.
\end{defn}

\begin{cor}\label{cor-perminterpstretched}
Suppose that $A \subseteq \{x \geq 1: x \equiv 1 \mod r\}$ and $B \subseteq \{x \geq 1 : x \equiv 0 \mod r\}$ with $A \cup B$ odd-ended and $1 \in A$.  
Define $(c_n)_{n \geq 0}$ via 
\[
\left(1-\sum_{a \in A} \frac{x^a}{a!} + \sum_{b \in B} \frac{x^b}{b!} \right)^{-1} = \sum_{n \geq 0} c_n \frac{x^n}{n!}.
\]
Then the values $c_n$ are all non-negative integers, and moreover have a natural combinatorial interpretation as  
\[
c_n = \sum_\sigma w_\sigma,
\]
where the sum is over all permutations $\sigma$ of $[n]$.

\end{cor}

See Section \ref{sec-run-theorem} for alternate combinatorial interpretations that are equivalent to the results in Corollaries \ref{cor-partitioninterpstretched} and \ref{cor-perminterpstretched} that follow from the Run Theorem.

\subsection{Recovering Gessel's Interpretation in the $r$-Stretched Case}\label{subsec-GesselStretched}
Taking the first $A \cup B = \{1,r,r+1,\ldots,(m-1)r,(m-1)r+1\}$, the $r$-stretched version of $[2m-1]$, we see that the fixed points are those with block sizes only given by $1$ and $(m-1)r+1$.  The case with $m=1$ is again trivial, and for $m>1$, by Corollary \ref{cor-perminterpstretched} the maximal increasing runs are made up of freeze skip $r$-tuples (with $mr$ elements in total) with a potential group of up to $r-1$ singleton blocks to the left, and so if there are $k$ freeze skip $r$-tuples, there must be $rmk+j$ elements in the maximal increasing run for some $0 \leq j \leq r-1$, and so they have length $kmr+j$ for some $0 \leq j \leq r-1$. This is Gessel's result \cite[Proposition 2.4]{Gessel}.

When $A \cup B = \{1,r,r+1,2r,\ldots\}$, we only have a single bottom element $1$.  By Corollary \ref{cor-perminterpstretched}, the maximal increasing runs are made up of up to $r-1$ singletons, and so $c_n$ counts the number of permutations of $[n]$ whose maximal increasing runs have length at most $r-1$.
As mentioned in the Introduction, this recovers a result of David and Barton (\cite[pp. 156-157]{DavidBarton}, see also for example \cite{GesselZhuang}).

\section{General conditions for non-negativity}\label{sec-necsuff}
In this section, we consider the necessary and sufficient conditions for non-negativity of $G_{A,B}(x)$ when $A \subseteq \mathcal{O}$ and $B \subseteq \mathcal{E}$ with $1 \in A$.  We prove a necessary condition from Theorem \ref{thm-necessary} and then indicate how our proof technique of fixing the underlying permutation is not able to prove Conjecture \ref{conj-sufficient}.

\subsection{Proof of Theorem \ref{thm-necessary}}

Suppose that $A \subseteq \mathcal{O}$ and $B \subseteq \mathcal{E}$ with $1 \in A$. First assume $A \cup B = [2m]$. 
We will consider the general case (where sizes at least $2m+2$ are also allowed in $A \cup B$) in the last paragraph of the proof. 

We focus on the value $c_{2m+2}$. Recall that the algorithm from the proof of Theorem \ref{thm-odd-ended} considers all ordered set partitions of $[2m+2]$ so that the block sizes are at most $2m$, and so that the elements in each block are in increasing order and so have an associated underlying permutation. From the algorithm, the term $c_{2m+2}$ counts the ordered set partitions of $[2m+2]$ that are fixed points of the involution, with the signs of each fixed partition based on the sizes of the blocks in the partition. So we count the signed fixed points of $i_{2m+2}$. We note that the proof of Theorem \ref{thm-odd-ended} showed that all fixed points came with a positive sign, but we are not odd-ended and so must consider the signs of the fixed points.  

Any block of size $2$ through $2m$, if considered by the algorithm (as in, it has not been skipped in a freeze skip pair), will split off a singleton block in Case D, and so such a ordered set partition will not be fixed by the algorithm. In particular, the fixed points of the algorithm must start with a singleton block $P_\ell$.

Further, any block of size $1$ will perform a non-skip move under the algorithm unless either:
\begin{enumerate}
    \item  its entry in the underlying permutation is in increasing order with the entry to the left, and the next block has size $2m$ (so those two blocks will be in a freeze skip pair), or 
    \item  its entry in the underlying permutation is not in increasing order with the entry to the left (so the singleton block will be skipped).
\end{enumerate}
We consider the possibilities separately.

If the entries of the singleton $P_{\ell}$ and the block $P_{\ell-1}$ to its left are in order, then $P_{\ell-1}$ must be of size $2m$, so  $P_{\ell-2}=P_1$ is a singleton.  The only restriction on the entries of the permutation is that the last $2m+1$ entries --- encompassing the blocks $P_{\ell-1}$ and $P_\ell$ --- must be increasing, so there are exactly $2m+2$ such ordered partitions into blocks with sizes $1$, $2m$, $1$, with the last two blocks having entries in the underlying permutation in increasing order.

If the entries of $P_{\ell}$ and $P_{\ell-1}$ are not in increasing order, then by Case B the singleton block $P_{\ell}$ is skipped. Again, the next block $P_{\ell-1}$ will split unless it is also singleton. By the same argument above, $P_{\ell-1}$ will do a freeze skip if its entries are in increasing order with the entries of $P_{\ell-2}$ and $P_{\ell-2}$ has size $2m$; the same count above gives $2m+1$ ways this can occur (note that the last two entries in the underlying permutation must not be increasing, by assumption).    

If instead the block $P_{\ell-1}$ has entries that are not in increasing order with those in $P_{\ell-2}$, then $P_{\ell-1}$ is skipped.  Now there are $2m$ elements remaining, 
and to be a fixed point of the involution, all remaining blocks must do block skips. Therefore all blocks are singletons whose entries are not in increasing order, and so $P$ is partitioned into $2m+2$ singleton blocks, and the underlying permutation is the decreasing permutation. 

We now consider the signs associated with these fixed points. Those partitions with a block with size $2m$ have a negative sign.  The partition of all singletons has a positive sign, and so $c_{2m+2} = -(2m+2)-(2m+1)+1 = -2(2m+1)<0$. 

\medskip

Lastly, suppose there are also blocks with size at least $2m+2$. 
Since any blocks with size $2m+2$ count with a negative sign and and all blocks with size at least $2m+3$ do not alter the count for partitions of $[2m+2]$, this shows that the coefficient $c_{2m+2}$ is negative in this case (with value unchanged if $2m+2 \notin A \cup B$, and value $c_{2m+2} = -2(2m+1)-1$ if $2m+2 \in A \cup B$).

\subsection{Sufficient Condition Limitations via Combinatorial Algorithm}
In this section we describe the limitations of our current proof technique in relation to Conjecture \ref{conj-sufficient}. 

Note that our algorithm fixes the underlying permutation. We will find a particular permutation for a particular $A \subseteq \mathcal{O}$ and $B \subseteq \mathcal{E}$ so that the net count for this permutation is negative, and so our method of finding non-negativity by proving it for each fixed permutation (via a sign-reversing involution that only fixes elements that count positively) cannot work.

Consider $A \cup B = \{1,4,6,8,10\ldots\}$ 
and $n=5$, and fix the permutation 12345.  Then only blocks with size $1$ and $4$ can be used, and there is one way to use all singleton blocks, namely $1/2/3/4/5$, whereas there are two ways to use a block with size $4$ on this fixed permutation: $1/2345$ and $1234/5$. So the net count for this fixed permutation is $-1$.  Any sign-reversing involution that fixes the underlying permutation will have a fixed point that counts negatively for this particular permutation.

We note that this does not imply that there are negative coefficients in the reciprocal series, since other permutations of $[5]$ have positive counts; it only shows that we cannot produce a proof where the underlying permutations stay the same and we have an involution leaving only positive blocks for each fixed permutation.

\subsection{Sufficient Condition via Analytic Techniques}
We next show analytically how Conjecture \ref{conj-sufficient} can be reduced to specific sets $A \cup B$ for each positive integer $m$.  We use this to prove Conjecture \ref{conj-sufficient} in the cases where $m=1$ and $m=2$.

To begin, we have the following general lemma.

\begin{lemma} \label{lem-conjreduce}
Suppose that $f(x)$ and $g(x)$ satisfy $f(0)=1$ and $g(0)=0$ and furthermore suppose that the series for $\frac{1}{f(x)}$ and the series for $g(x)$ have all coefficients non-negative.  

Then the coefficients of the series for $\frac{1}{f(x)-g(x)}$ are also non-negative. 
\end{lemma}

\begin{proof}
We compute
\[
\frac{1}{f(x)-g(x)} = \frac{1}{f(x)}\cdot \frac{1}{1-g(x)/f(x)} = \left( \frac{1}{f(x)} \right) \left(\sum_{k=0}^{\infty} \left( g(x) \cdot \frac{1}{f(x)} \right)^k\right),
\]
and by assumption both $\frac{1}{f(x)}$ and $g(x) \cdot \frac{1}{f(x)}$ have all coefficients non-negative.  This completes the proof.
\end{proof}

This leads to the following.

\begin{cor} \label{cor-sufficientreduce}
    Suppose that $A \subseteq \mathcal{O}$ and $B \subseteq \mathcal{E}$ with $1 \in A$, and suppose that the smallest positive number in $\{1,2,3,\ldots\} \setminus (A \cup B)$ is $2m$.  

    If the series for 
\[
\frac{1}{1+ \sum_{k=1}^{2m-1} \frac{(-1)^k x^{k}}{k!} + \sum_{k=m+1}^{\infty} \frac{x^{2k}}{(2k)!}},
\]
corresponding to coefficients in $[2m-1] \cup \{2m+2,2m+4,\ldots\}$, has all coefficients non-negative, then the series for 
\[
\frac{1}{1-\sum_{a \in A} \frac{x^a}{a!} + \sum_{b \in B} \frac{x^b}{b!}}
\]
has all coefficients non-negative.
\end{cor}

\begin{proof}
For each positive $m \geq 1$ let 
\[
f_m(x)
=1+ \sum_{k=1}^{2m-1} \frac{(-1)^k x^{k}}{k!} + \sum_{k=m+1}^{\infty} \frac{x^{2k}}{(2k)!}.
\]
By assumption, suppose also that $\frac{1}{f_m(x)}$ has all non-negative coefficients. 

Consider now the arbitrary $A \cup B$ with $1 \in A$ and with $2m$ the smallest positive integer missing from $A \cup B$. Let $C$ be the set of even numbers, strictly greater than $2m$, that are {\it missing from} $B$, together with the set of odd numbers, strictly greater than $2m$, that are {\it in} $A$. Then
\[
\frac{1}{1 - \sum_{a \in A} \frac{x^a}{a!} + \sum_{b \in B} \frac{x^b}{b!} }  =   \frac{1}{f_m(x) - \sum_{c \in C} \frac{x^c}{c!}}
\]
has all coefficients non-negative by Lemma \ref{lem-conjreduce}.
\end{proof}

So to prove Conjecture \ref{conj-sufficient}, by Corollary \ref{cor-sufficientreduce} it suffices to consider 
\[
A \cup B = \{1,2,\ldots,2m-1,2m+2,2m+4,\ldots\}
\]
for each positive integer $m$.  This leads to the following proof of Conjecture \ref{conj-sufficient} when $m=1$.

\begin{thm}\label{thm-analytic}
Let $f(x) = 1 - x + \sum_{k \mbox{ even, } k \geq 4} x^k/k!$.  Then the reciprocal $\frac{1}{f(x)}$ has a series with all non-negative coefficients.
\end{thm}

\begin{proof}
Let $a>0$.  We'll show that there is a value for $a$ so that $f(x) g(x) = 1 -p(x)$, where $g(x)  = 1 + a x + a^2 x^2 + \cdots$ and $p(x)$ is a non-negative series with $p(0)=0$.  Then the reciprocal $1/f(x) = g(x) \cdot (1/(1-p(x)))$ is a product of two non-negative series and is hence non-negative.

Let $n$ be a positive integer.  Then

\begin{eqnarray*} [x^n] f(x) g(x) &=& a^n  - a^{n-1} + \sum_{k \mbox{ even, } 4 \leq k \leq n} a^{n-k}/k!\\
&=& a^n \left(1  - 1/a + \sum_{k \mbox{ even, } 4 \leq k \leq n} (1/a)^k/k! \right)\\
&\leq& a^n f(1/a).
\end{eqnarray*}
If we find $a>0$ so that $f(1/a) < 0$, then $[x^n]f(x)g(x) \leq 0$ for all $n \geq 1$ as desired. 
But $f(x) = \frac{e^{x}+e^{-x}}{2}-x-x^2/2$ and $f(2)<0$ so we may take $a = 1/2$.
\end{proof}

With similar reasoning, we have a proof when $m=2$.

\begin{thm}\label{thm-analytic2}
Let $f(x) = 1 - x +\frac{x^2}{2!} - \frac{x^3}{3!} + \sum_{k \mbox{ even, } k \geq 6} x^k/k!$.  Then the reciprocal $\frac{1}{f(x)}$ has a series with all non-negative coefficients.
\end{thm}

\begin{proof}
Let $f(x) = 1 - x +\frac{x^2}{2!} - \frac{x^3}{3!} + \sum_{k \mbox{ even, } k \geq 6} \frac{x^k}{k!}$. 
We'll again show that there is an $a>0$ so that $f(x)g(x) = 1-p(x)$, where now $g(x) = 1+x+\frac{x^2}{2!}+\frac{x^3}{3!} + a^4x^4 + a^5x^5 + \cdots$ and $p(x)$ is a non-negative series with $p(0)=0$, which as before will imply that $\frac{1}{f(x)}$ has a series with all non-negative coefficients.  It will turn out that $a=1/2$. 

We compute $[x^n]f(x)g(x)$.  Note that these series start as truncated versions of $e^{-x}$ and $e^{x}$ through the degree $3$ term, so clearly $[x^0]f(x)g(x)=1$ and $[x^n]f(x)g(x)=0$ for $n=1,2,3$.  Then
\begin{eqnarray*}
\left[x^4\right]f(x)g(x) &=& a^4 - \frac{1}{3!} + \frac{1}{2!\cdot 2!} - \frac{1}{3!}=a^4-\frac{2}{4!}\\
\left[x^5\right] f(x) g(x) &=& a^5 - a^4 + \frac{1}{3! 2!} - \frac{1}{2! 3!} = a^4(a-1)\\
\left[x^6\right]f(x)g(x) &=& a^6-a^5+\frac{a^4}{2!} - \frac{1}{3!3!} + \frac{1}{6!}\\
\left[x^7\right]f(x)g(x) &=& a^7 - a^6 + \frac{a^5}{2!} - \frac{a^4}{3!} + \frac{1}{6!} \\
\left[x^{2n}\right]f(x)g(x) &\leq& a^{2n} \left(f(1/a) \right) + \frac{1}{2!(2n-2)!} + \frac{1}{(2n)!} \qquad \text{($n \geq 4$)}\\
\left[x^{2n+1}\right]f(x)g(x) &\leq& a^{2n+1} \left( f(1/a) \right) + \frac{1}{3!(2n-3)!} + \frac{1}{(2n-1)!} \qquad \text{($n \geq 4$).}
\end{eqnarray*}

Now, note $f(x) = \cosh(x) - x - \frac{x^3}{3!} - \frac{x^4}{4!}$. Also, it can be checked that when $a=1/2$ we have that the $n=4$ through $n=7$ coefficients  in the above computations are negative. For that choice of $a$ we have $f(1/a) = f(2) = -0.2378043089...$ and so $a^x f(1/a) \leq a^x \cdot (-1/5)$.  Now
\[
a^8 f(1/a) + \frac{1}{2!6!} + \frac{1}{8!} \leq \left(\frac{1}{2}\right)^8 \cdot \left( \frac{-1}{5}\right) + \frac{4 \cdot 7 + 1}{8!} = \frac{-3\cdot 3 \cdot 7/2+29}{8!}<0
\]
and
\[
a^9 f(1/a) + \frac{1}{3!6!} + \frac{1}{8!} \leq \left(\frac{1}{2}\right)^9 \cdot \left(\frac{-1}{5}\right) + \frac{56+6}{3!8!} = \frac{-3 \cdot 3 \cdot 7 \cdot 3/2+62}{3!8!}<0.
\]
For even values larger than $9$, we inductively assume $a^{2n}f(1/a) + \frac{1}{2!(2n-2)!} + \frac{1}{(2n)!}<0$ to show the result for $a_{2n+2}$:

\[
a^{2n+2}f(1/a)+\frac{1}{2!(2n)!} + \frac{1}{(2n+2)!} \leq \left( \frac{1}{4} \right) \left( a^{2n}f(1/a) + \frac{1}{2!(2n-2)!} + \frac{1}{(2n)!} \right)<0.
\]
For odd values larger than $9$, we also inductively have 
\[
a^{2n+3}f(1/a)+\frac{1}{3!(2n-1)!} + \frac{1}{(2n+1)!} \leq \left( \frac{1}{4} \right) \left( a^{2n+1}f(1/a) + \frac{1}{2!(2n-3)!} + \frac{1}{(2n-1)!} \right)<0.
\]
This completes the proof.
\end{proof}

Unfortunately, this technique will not work for all values of $m$,
even if we allow ourselves flexibility with the parameter $a$.  Computations show that with $f(x) = \cosh(x)-x-x^3/3!-x^5/5!-\ldots-x^{2m-1}/(2m-1)!-x^{2m}/(2m)!$, with $m=5$,
and
$g(x)=1+x+x^2/2!+...+x^{2m-1}/(2m-1)! + a^{2m}x^{2m}/(1-ax)$, there is no choice of $a$ for which $f(x)g(x)=1-p(x)$ with $p(0)=0$ and all coefficients of $p(x)$ non-negative. 

\section{Proofs Using the Run Theorem} \label{sec-run-theorem}

We will now show that the ``Run Theorem'' (Theorem \ref{thm-run-theorem} below), a theorem in noncommutative symmetric function theory, and its power series counterpart (Corollary \ref{cor-run-theorem} below), can be used to prove the conclusion of non-negativity in Theorem \ref{thm-odd-endedstretched} (and its $b=1$, $r=2$ case, Theorem \ref{thm-odd-ended}) and can also be used to derive an alternate combinatorial description of the coefficients of the reciprocal.

The field of noncommutative symmetric functions was first developed in \cite{Gelfandetal}. For the theory we need we follow the exposition that is given in \cite{GesselZhuang}. Let $F$ be a field of characteristic $0$ and let $A = F \langle \langle X_1, X_2, \ldots \rangle \rangle$ be the algebra of formal power series in the noncommuting variables $X_i$ for $i \geq 1$.

Let $n \geq 0$ be an integer.  We say that $L$ is a composition of $n$, written $L \models n$, if $L$ is a sequence $L = (L_1, \ldots, L_k)$ of positive integers that sum to $n$ where $k \geq 0$ is an integer.  Note that the only composition of $0$ is the empty composition, the $0$-tuple $L=()$.  

We will now define the subalgebra $\Sym$ of $A$ of noncommutative symmetric functions. Let \[{\bf h}_n = \sum_{i_1 \leq i_2 \leq \cdots \leq i_n} X_{i_1} X_{i_2} \cdots X_{i_n}\] be the complete noncommutative symmetric function.  For any composition $L = (L_1, \ldots, L_k)$, let ${\bf h}_L = {\bf h}_{L_1} \cdots {\bf h}_{L_k}$.   We take ${\bf h}_L = 1$ for the empty composition $L = ()$. $\Sym$ is defined to be the subalgebra of $A$ consisting of all (possibly infinite) linear combinations of the ${\bf h}_L$ with coefficients in $F$. By this definition, $\Sym$ has basis $H = \{{\bf h}_L : L \mbox{ a composition}\}$ when it is viewed as a vector space over $F$.

If $s = (i_1, \ldots, i_n)$ is a sequence of positive integers, a run of $s$ is a maximal non-decreasing subsequence of consecutive entries of $s$.  Let $L(s) \models n$ be the run composition of $s$, i.e. $L=(L_1, \ldots, L_k)$ where $L_i$ is the length (i.e. number of elements) of the $i$th run in $s$.  For any composition $L \models n$, we define the ribbon non-commutative symmetric function \[{\bf r}_L = \sum_{s=(i_1, \ldots, i_n) : L(s)=L} X_{i_1} X_{i_2} \cdots X_{i_n}.\]  As shown in \cite{GesselZhuang}, $R = \{{\bf r}_L : L \mbox{ a composition}\}$ is also a basis for $\Sym$.

\begin{thm}[The Run Theorem] \label{thm-run-theorem}
Let $F$ be a field of characteristic $0$. 
Let $(a_i)_{i=0}^\infty$ be a sequence in $F$ with $a_0=1$. In $\Sym$, the algebra of noncommutative symmetric functions over $F$, we have
\[\left(\sum_{n=0}^\infty a_n {\bf h}_n \right)^{-1} = \sum_L w_L {\bf r}_L\]
where the sequence $(w_i)_{i=0}^\infty$ in $F$ is defined by the equation 
\[\left(\sum_{n=0}^\infty a_n x^n \right)^{-1} = \sum_{n=0}^\infty w_n x^n\]
in the ring $F[[x]]$ of power series in $x$ with coefficients in $F$, the sum on the right is over all compositions $L$, and $w_L = w_{L_1} \cdots w_{L_k}$ for all compositions $L=(L_1,\ldots,L_k)$.

\end{thm}

The Run Theorem was first proved in an equivalent form in \cite[Theorem 5.2]{Gesselthesis}.  The form we give here follows the presentations given in \cite[Theorem 1]{Zhuang} and \cite[Theorem 11]{GesselZhuang}.  Note that the coefficients $a_i$ and $w_i$ may actually be taken to lie in a unital $F$-algebra (see \cite[Theorem 1]{Zhuang}), but the form we give is sufficient for our needs.

As noted in \cite[Section 4.2]{GesselZhuang}, there is a homomorphism $\Phi:\Sym \to F[[x]]$ for which $\Phi({\bf h}_n) = x^n/n!$ and $\Phi({\bf r}_L) = \beta(L) x^n/n!$ where $\beta(L)$ is the number of permutations $\pi$ in $S_n$ with run-composition $L(\pi)=L$, i.e. with consecutive run-lengths as given by $L$.  Applying this $\Phi$ to the reciprocal relation in $\Sym$ in Theorem \ref{thm-run-theorem} gives the following power series version of the Run Theorem.

\begin{cor} \label{cor-run-theorem}
Let  $F$ be a field of characteristic $0$.  Let $(a_i)_{i=0}^\infty$ be a sequence in $F$ with $a_0=1$. In the ring of power series $F[[x]]$ we have
\[\left(\sum_{n=0}^\infty a_n \frac{x^n}{n!} \right)^{-1} = \sum_{n=0}^\infty b_n  \frac{x^n}{n!}\]
where
\[b_n  = \sum_{L \models n} w_L \beta(L),\] $w_L = w_{L_1} \cdots w_{L_k}$ for all compositions $L=(L_1,\ldots,L_k)$, and the sequence $(w_i)_{i=0}^\infty$ in $F$ is defined by the equation 
\[\left(\sum_{n=0}^\infty a_n x^n \right)^{-1} = \sum_{n=0}^\infty w_n x^n\]
in $F[[x]]$. 
\end{cor}

We can use Corollary \ref{cor-run-theorem} to give another proof of Theorem \ref{thm-odd-endedstretched} with an equivalent combinatorial formula for the coefficient of the reciprocal.  We first need the following standard result on reciprocals in $F[[x]]$.

\begin{thm} \label{thm-ogf-reciprocal}
Let $F$ be a field of characteristic $0$.  Let $(a_i)_{i=0}^\infty$ be a sequence in $F$ with $a_0=1$. Then 
\[\left(\sum_{n=0}^\infty a_n x^n \right)^{-1} = \sum_{n=0}^\infty w_n x^n,\]
where
\[w_n = \sum_{L \models n} (-1)^{|L|} a_L,\]
and $a_L = a_{L_1} \cdots a_{L_k}$ and $|L|=k$ for each composition $L = (L_1, \ldots, L_k)$.
\end{thm}
\begin{proof}
We have
\begin{align*}
\left(\sum_{n=0}^\infty a_n x^n \right)^{-1} & = \left( 1 - \sum_{\ell=1}^\infty (-a_\ell) x^\ell \right)^{-1} \\
& = 1 + \sum_{k=1}^\infty \left(\sum_{\ell=1}^\infty (-a_\ell) x^\ell \right)^k \\
& = 1 + \sum_{k=1}^\infty \sum_{(L_1, \ldots, L_k)} (-1)^k a_{L_1} \cdots a_{L_k} x^{L_1+ \cdots + L_k}\\
& = 1 + \sum_{n=1}^\infty \left(\sum_{L \models n } (-1)^{|L|} a_L\right) x^n
\end{align*}
We use the formula for the geometric series in $F[[x]]$ for the second equality above.  Note that $w_0=1$ if we adopt the usual convention that the only composition of $0$ is the empty composition $L=()$ and that $a_L=1$ for this composition.
\end{proof}
In particular, we use the following corollary of Theorem \ref{thm-ogf-reciprocal}.
\begin{cor} \label{cor-ogf-reciprocal}
Let $F$ be a field of characteristic $0$. Let $S$ be a subset of the positive integers.  Then
\[\left(1 - \sum_{s \in S} x^s\right)^{-1} = \sum_{n=0}^\infty w_n x^n,\] where $w_n$ is the number of compositions of $n$, all of whose block sizes lie in $S$.
\end{cor}
\begin{proof}
We apply Theorem \ref{thm-ogf-reciprocal} with $a_0=1$ and for all positive integers $n$, $a_n = -1$ if $n \in S$ and $a_n = 0$ if $n \not \in S$.  Then for any composition $L = (L_1, \ldots, L_k)$, $(-1)^{|L|} a_L = 1$ if $L$ is a composition, all of whose blocks are in $S$, and $(-1)^{|L|} a_L = 0$ otherwise.  Thus $w_n$ is the number of compositions of $n$, all of whose block sizes lie in $S$.
\end{proof}

\begin{proof}[Proof of Theorem \ref{thm-odd-endedstretched}]
Let ${\cal O}$ be the set of odd positive integers and ${\cal E}$ be the set of even positive integers.  We define the function $n^*$ on the non-negative integers by setting $n^*(k) = krb/2$ if $k$ is even, and $n^*(k)=((k-1)r/2 +1)b$ if $k$ is odd.  Thus $n^*(0)=0$, $n^*({\cal O})=\{(jr+1)b : r \geq 0,j \geq 0\} = {\cal O}^*$, and $n^*({\cal E})=\{jrb : r \geq 1,j\geq 1\} = {\cal E}^*$.

Suppose $A \subseteq {\cal O}^*$, $B \subseteq {\cal E}^*$, $b \in A$, and $A \cup B$ is odd-ended.  Then $A \cup B$ is the image under the map $n^*$ of the union of a set of finitely or infinitely many intervals of integers of the form $I_1 = [c_1,d_1]$, $I_2 = [c_2,d_2], \ldots$ satisfying the following conditions:
\begin{enumerate}
    \item[(i)] $c_1=0$ but otherwise $c_i$ and $d_i$ are odd for all $i$, 
    \item[(ii)] $0 = c_1 < d_1 < c_2 \leq d_2 < c_3 \leq d_3 < \cdots$,  
    \item[(iii)] if there are finitely many intervals, the last interval $I_m$ may also be of the form $I_m = [c_m, \infty)$.
\end{enumerate}

Let
\[ f(x) = 1 - \sum_{a \in A} \frac{x^a}{a!} + \sum_{b \in B} \frac{x^b}{b!} = \sum_k  \sum_{c_k \leq  i \leq d_k} (-1)^i \frac{x^{n^*(i)}}{n^*(i)!}\]
and
\[ g(x) = 1 - \sum_{a \in A} x^a + \sum_{b \in B} x^b = \sum_k  \sum_{c_k \leq  i \leq d_k} (-1)^i x^{n^*(i)}.\]
Briefly, we will use algebraic manipulations to show that ordinary generating function (o.g.f.) $g(x)$ has a reciprocal whose o.g.f. has non-negative coefficients $w_n$.  Corollary \ref{cor-run-theorem} then implies that the exponential generating function (e.g.f.) $f(x)$ has a reciprocal whose e.g.f. has non-negative coefficients $b_n$.  We will use Corollary \ref{cor-ogf-reciprocal} to show that $w_n$ counts the number of compositions of a certain form.  Thus the formula for $b_n$ given by Corollary \ref{cor-run-theorem} will be a count of permutations that are combinatorially weighted according to their run-lengths.  We now fill in the details.

Let $h(x) = \sum_{k=0}^\infty (x^{krb} - x^{(kr+1)b})$.  Note that
\begin{align*}
h(x) &= (1-x^b) \sum_{k=0}^\infty x^{krb}\\
&= (1-x^b)/(1-x^{rb}) \\
& = 1/(1 + x^b + x^{2b} + \cdots + x^{(r-1)b}).
\end{align*}
Thus we have
\[1/h(x) = 1 + x^b + x^{2b} + \cdots + x^{(r-1)b}\]
and
\[1 - 1/h(x) = - (x^b + x^{2b} + \cdots + x^{(r-1)b}).\]

Note that
\[  \sum_{c_1 \leq  i \leq d_1} (-1)^i x^{n^*(i)} = h(x) \left(1 - x^{n^*(d_1+1)} \right)\] and that for all $k \geq 2$
\begin{align*}
\sum_{c_k \leq  i \leq d_k} (-1)^i x^{n^*(i)} & = h(x) \left(x^{n^*(c_k-1)} - x^{n^*(d_k+1)}\right)  - x^{n^*(c_k-1)} \\
& = h(x) \left( (1-1/h(x)) x^{n^*(c_k-1)} -  x^{n^*(d_k+1)} \right) \\
&= h(x) \left( - (x^b + x^{2b} + \cdots + x^{(r-1)b } ) x^{n^*(c_k-1)} -  x^{n^*(d_k+1)} \right)
\end{align*}
Thus $g(x) = h(x)(1 - \sum_{s \in S} x^s )$ where
\begin{eqnarray*}
S &=& \{n^*(c_k-1) + j b : k \geq 2, 1 \leq j \leq r-1\} \cup \{n^*(d_k+1) : k \geq 1\}\\
&=& \{n^*(c_k)+jb: k \geq 2,0\leq j \leq r-2\} \cup \{n^*(d_k+1):k \geq 1\}.
\end{eqnarray*}
So
\[1/g(x)  = (1 + x^b + x^{2b} + \cdots x^{(r-1)b} ) \left(1 - \sum_{s \in S} x^s \right)^{-1}.\]
By Corollary \ref{cor-ogf-reciprocal}, \[1/g(x) = \sum_{n \geq 0} w_n x^n,\] 
where $w_n$ is the number of compositions of $n$ whose block sizes must lie in $S$, where we also allow the first block to have size $b, 2b, \ldots, (r-1)b$.  Clearly $w_n$ is non-negative. 

We then apply Corollary \ref{cor-run-theorem} to get that
\[1/f(x) = \sum_{n \geq 0} b_n \frac{x^n}{n!}\]
where \[b_n = \sum_{L \models n} w_L \beta(L).\]
Clearly $b_n$ is also non-negative.  This formula can also be considered a combinatorial formula for $b_n$.  One obtains $b_n$ by summing over all permutations $\pi$ of $[n]$ so that each permutation $\pi$ is weighted by $w_L  = w_{L_1} \cdots w_{L_k}$ if it has run-composition $L(\pi) = (L_1, \ldots, L_k)$.
\end{proof}
Note that this proof of Theorem \ref{thm-odd-endedstretched} appeals to the theory of non-commutative symmetric functions as well as algebraic manipulations of generating functions whereas the proofs of Theorems \ref{thm-odd-ended} and \ref{thm-odd-endedstretched} given in Sections \ref{sec-pfofodd-ended} and \ref{sec-stretched} are purely combinatorial.   Those proofs also appeal to sign-reversing involutions, a tactic used in \cite{EGS} to give non-negativity results and combinatorial formulas for the coefficients of the \emph{compositional inverses} of a large class of power series.  

Note also that the combinatorial count given by the Run Theorem is of a slightly different form that the ones given in Corollaries \ref{cor-partitioninterpstretched} and \ref{cor-perminterpstretched}. The translation from this interpretation to that in Corollary \ref{cor-perminterpstretched} amounts to considering up to $r-1$ singleton blocks at the end of the algorithm (which are fixed), along with viewing freeze skip pairs as a single block (with size $n^*(d_k+1)$), and viewing the blocks with at most $r-2$ singletons to the right as another single block (with size $n^*(c_k)+jb$).

Lastly, we note that it is not possible to use the Run Theorem to answer Conjecture \ref{conj-sufficient} as the corresponding o.g.f.s often have some negative coefficients.

\medskip

{\bf Acknowledgements:} We thank Ira Gessel for helpful discussions and for suggesting the use of the Run Theorem.

\bibliographystyle{alphaurl}
\bibliography{4-arxiv-submission-2}

\end{document}